\documentclass{amsart}
\usepackage[all]{xy}
\usepackage{graphicx}
\usepackage{amsmath}
\usepackage{amsfonts}
\usepackage{amssymb}

\newtheorem{theorem}{Theorem}[section]

\newtheorem{corollary}[theorem]{Corollary}

\newtheorem{proposition}[theorem]{Proposition}

\theoremstyle{definition}
\newtheorem{definition}[theorem]{Definition}

\theoremstyle{remark}
\theoremstyle{remark}
\newtheorem{remark}[theorem]{Remark}

\numberwithin{equation}{section}

\newcommand{\Rr}{\mathbb R}
\newcommand{\Qq}{\mathbb Q}
\newcommand{\Z}{\mathbb{Z}}
\newcommand{\T}{\mathbb{T}}

\newcommand{\g}{\ensuremath{\mathfrak{g}}}                

\newcommand{\pb}[1]{\left\{#1\right\}}                   
\renewcommand{\d}{\mathrm d}
\newcommand{\grad}{\mathrm{grad}\,}
\DeclareMathOperator{\Ad}{Ad}                             
\DeclareMathOperator{\rank}{rank}                             
\newcommand{\X}{\ensuremath{\mathfrak{X}}}                
\newcommand{\Ha}{\ensuremath{X}}                

\begin{document}
\title[A survey on cosymplectic geometry]{A survey on cosymplectic geometry}

\author[B. Cappelletti-Montano]{Beniamino Cappelletti-Montano}
 \address{Dipartimento di Matematica e Informatica, Universit\`a degli Studi di
 Cagliari, Via Ospedale 72, 09124 Cagliari}
 \email{b.cappellettimontano@gmail.com}

\author[A. De Nicola]{Antonio De Nicola}
 \address{CMUC, Department of Mathematics, University of Coimbra, Apartado 3008, 3001-454 Coimbra, Portugal}
 \email{antondenicola@gmail.com}

\author[I. Yudin]{Ivan Yudin}
 \address{CMUC, Department of Mathematics, University of Coimbra, Apartado 3008, 3001-454 Coimbra, Portugal}
 \email{yudin@mat.uc.pt}



\thanks{Research partially supported by CMUC, funded by the European program
COMPETE/FEDER, by FCT (Portugal) grants PEst-C/MAT/UI0324/2011, PTDC/MAT/099880/2008 (A.D.N.), SFRH/BPD/31788/2006 (I.Y.), by MICINN (Spain) grant
MTM2009-13383 (A.D.N.), and by Prin 2010/11 -- Variet\`{a} reali e complesse: geometria, topologia e analisi armonica –- Italy (B.C.M.).}

\maketitle

\begin{abstract}
We give an up-to-date overview of geometric and topological properties of
cosymplectic and coK\"{a}hler manifolds. We also mention some of their applications
to time-dependent mechanics.
\end{abstract}

\setcounter{tocdepth}{2} \tableofcontents

\section{Introduction}

In recent years there has been an increasing interest towards almost
contact geometry and related topics, both from the pure geometrical
point of view and due to its applications in wide areas of Physics.
In particular, the odd dimensional manifolds proved to be very
important for some physical theories like supergravity and M-theory.

An important class of almost contact manifolds is given by
cosymplectic manifolds. By a \emph{cosymplectic manifold} in this paper
we mean  a smooth manifold of dimension $2n+1$ endowed with a closed
$1$-form $\eta$ and a closed $2$-form $\omega$ such that
$\eta\wedge\omega^{n}$ is a volume form. This is the original
definition of cosymplectic manifold, due to Libermann and dating back
to 1958.
Later Blair used the term ``cosymplectic''
with a different meaning,  namely for denoting
cosymplectic manifolds endowed with a compatible almost contact
metric structure satisfying a normality condition.
Despite the wide use  of Blair's terminology, we will call the latter
manifolds \emph{coK\"ahler}, which is  becoming a common  practice now.

Also, there is a weaker notion of almost coK\"{a}hler manifold
(called sometimes almost cosymplectic!) which is just a cosymplectic
manifold with a compatible almost contact metric structure, without
the normality condition. Note that every cosymplectic manifold
admits an almost coK\"{a}hler structure.

Since the foundational papers of Libermann and Blair, there have been
many advances in the subject and its related topics. Several papers
dealing with cosymplectic or coK\"{a}hler manifolds appeared in the
last 40 years,  even though in a rather sporadic way. Unfortunately,
there is neither a monograph nor a paper which gives an exposition of the state of art
of the theory. In fact the recent very well-written monographs on
contact geometry -- the Blair's and Boyer-Galicki's books
\cite{blairbook2010}, \cite{galickibook} -- treat, from
different points of view, just contact metric manifolds (including
Sasakian manifolds), leaving for cosymplectic geometry not much
space.

We hope that the present survey can help to fill this gap and to
give a unified view of the subject.

The importance of cosymplectic manifolds for the geometric
description of time-dependent mechanics is nowadays widely
recognized. In most of the numerous formulations of time-dependent
mechanics cosymplectic manifolds do play a major role: either in the
classical descriptions
of regular Lagrangian systems (see e.g. \cite{deleon-m96}) and
of Hamiltonian systems \cite{deleon-m96,lacirasella12}, or in the
more elegant descriptions \'a la Tulczyjew
(\cite{deleon-m93,guzman10}) in terms of Lagrangian submanifolds. It
is far outside the scope of our presentation to make a detailed
review of these approaches or even a complete list of papers. The
1991 paper \cite{echeverria91} contains a comparative review at that
time of various equivalent approaches to the geometric treatment of
regular time-dependent Hamiltonian systems. For further information
about more recent literature on the geometric formulations of
time-dependent mechanics see \cite{barbero08, deleon-m96,
lacirasella12} and references therein. Here we just briefly
summarize, as an example, how cosymplectic manifolds appear in the
geometric setup of time-dependent Lagrangian mechanics
\cite{barbero08,deleon-m96}. In such a model the configuration
manifold of autonomous Lagrangian mechanics is replaced by a
fibration $\pi:M\to \Rr$ and the jet bundle $J^1\pi$ of the local
sections of $\pi$ is the velocity phase space of the system. In the
simplest case the fibration is trivial, i.e. $M\cong Q\times \Rr$
where $Q$ is a manifold, and one has $J^1\pi\cong TM\times \Rr$. In
general the dynamics is controlled by a Lagrangian density
${\mathcal L}\in \Omega^1(J^1\pi)$ which is a semibasic 1-form. The
Poincar{\'e}-Cartan forms associated with the Lagrangian density
${\mathcal L}$ are defined using the ``vertical endomorphism'' $S$
of the bundle $J^1\pi$ (see \cite{echeverria91,saunders89}),
\begin{equation*}
    \Theta_{\mathcal L}=i_S\d{\mathcal L}+{\mathcal L},\qquad \Omega_{\mathcal L}=-\d\Theta_{\mathcal L}.
\end{equation*}
Geometrically, the Lagrangian density is regular if and only if the
pair $(\d t,\Omega_{\mathcal L})$ is a cosymplectic structure on the
manifold $J^1\pi$, where by $\d t$ we denote (with a slight abuse
of notation) the pull-back on $J^1\pi$  of the canonical volume
1-form of $\Rr$. Then, in the regular case, the dynamics of the
system is controlled by the integral curves of the Reeb vector field
$\xi_{\mathcal L}$. In the singular case, the structure $(\d
t,\Omega_{\mathcal L})$ is not cosymplectic and the dynamics is
controlled by an implicit differential equation, not necessarily
integrable.

We devote the first part of the survey to the basic properties of
cosymplectic structures. An important role in the theory is played
by the  reduction theory, which was developed by Albert in
\cite{albert89} in order to be able to consider the reduction of
time-dependent Hamiltonian systems with symmetry groups. The
reduction procedure was extended to the singular case by de Le\'{o}n
and Tuynman \cite{deleon96}. Then, in \cite{cantrijn02} the Albert
reduction theorem was adapted to deal with the case of a
time-dependent mechanical system with constraints.

As we pointed out before, any cosymplectic manifold can be endowed
with a Riemannian metric which makes it an almost coK\"{a}hler
manifold. We thus study, in the second part of the paper, the
geometry of almost coK\"{a}hler manifolds. We collect the most
important results of the theory and the various ways for
constructing explicit examples (by suspensions, by using Lie groups, as
hypersurfaces of K\"{a}hler manifolds, etc). Due to the
importance of the subject, a large part is dedicated to the study of the curvature
properties of almost coK\"{a}hler manifolds and to the topology of
cosymplectic and coK\"{a}hler manifolds. This last topic has become
of great interest in last years, after the very remarkable work
\cite{chinea-deleon-marrero93} of Chinea, de Le\'{o}n, Marrero  on
the topology of compact coK\"{a}hler manifolds. Recently Li
(\cite{li08}) and then Bazzoni and Oprea \cite{bazzoni-oprea}
discovered new methods for studying the topology of coK\"{a}hler
manifolds, which allow us to re-obtain many of the Chinea - de Le\'{o}n
- Marrero's results.

We also try to explain the several interplays  with K\"{a}hler geometry, which make coK\"{a}hler manifolds the most natural odd-dimensional counterpart of
K\"{a}hler manifolds. This becomes even more evident when one passes to the setting of $3$-structures. Indeed, while both coK\"{a}hler and Sasakian
manifolds admit a transversal K\"{a}hler structure, at the level of $3$-structures only the former  admit a transversal hyper-K\"{a}hler  structure.

There are many further topics related to cosymplectic geometry, that
have found almost no room in this survey due to the lack of space.
We mention only few of them: the study of submanifolds of
cosymplectic / coK\"{a}hler manifolds, the theory of harmonic maps,
Einstein-Weyl structures, and several remarkable generalizations
such as locally conformal cosymplectic manifolds, $k$-cosymplectic
manifolds or generalized almost contact structures. All these
subjects, which in turn often have relations with each other and
with Physics, show how rich and wide is the field which we call
``cosymplectic geometry''.
\\
\\
\textbf{Acknowledgement.} \
The authors thank the anonymous referees for the careful reading and for having improved the paper.

\section{Cosymplectic structures}

\subsection{Basic definitions and properties}
Cosymplectic manifolds were introduced by Libermann (see
\cite{libermann59,libermann62}).
\begin{definition}
An \emph{almost cosymplectic structure} on a manifold $M$ of odd
dimension $2n+1$ is a pair $(\eta,\omega)$, where $\eta$ is a
$1$-form and $\omega$ is a  $2$-form such that $\eta\wedge\omega^n$
is a volume form on $M$. The structure is said to be
\emph{cosymplectic} if $\eta$ and $\omega$ are closed.
\end{definition}
 If $(\eta,\omega)$ is
an almost cosymplectic structure on $M$, the triple
$(M,\eta,\omega)$ is said to be an  \emph{almost cosymplectic
manifold}. Similarly, a \emph{cosymplectic manifold} is a manifold
$M$ endowed with a cosymplectic structure $(\eta,\omega)$. In the
language of $G$-structures, an almost cosymplectic structure can be
defined equivalently as an $1\times Sp(n, \Rr)$-structure.

Among the early studies of cosymplectic manifolds, we mention those
of Lichnerowicz \cite{lichnerowicz61,lichnerowicz63}, Bruschi
\cite{bruschi63}, Takizawa \cite{takizawa63} and Okumura
\cite{okumura65}. Lichnerowicz studied the  Lie algebra of
infinitesimal automorphisms of a cosymplectic manifold, in analogy
with the symplectic case. Bruschi investigated homogeneous
cosymplectic spaces. Takizawa reviewed the Lagrange brackets and the
infinitesimal transformations. The work of Okumura will be described
in the next section since it actually belongs to the coK\"{a}hler
domain.

Every almost cosymplectic structure $(\eta,\omega)$ on $M$ induces
an isomorphism of $C^{\infty}(M)$-modules
$\flat_{(\eta,\omega)}:\X(M)\to \Omega^1(M)$ defined by
\begin{align}\label{bemolle}
 \flat_{(\eta,\omega)}(X)=i_X\omega+\eta(X)\eta,
\end{align}
for every vector field $X\in \X(M)$. A vector bundle isomorphism
(denoted with the same symbol) $\flat_{(\eta,\omega)}:TM\to T^*M$ is
also induced. Then the vector field
$\xi=\flat^{-1}_{(\eta,\omega)}(\eta)$ on $M$ is called the
\emph{Reeb vector field \label{reeb}} of
the  almost cosymplectic manifold
$(M,\eta,\omega)$ and  is
characterized by the following conditions
\[
 i_\xi\omega=0,\qquad \eta(\xi)=1.
\]
Conversely, we have the following characterization of almost
cosymplectic manifolds that follows from \cite[Proposition
2]{albert89}.

\begin{proposition}
Let $M$ be a manifold endowed with a $1$-form $\eta$  and a $2$-form
$\omega$ such that the map $\flat_{(\eta,\omega)}:\X(M)\to
\Omega^1(M)$ defined by \eqref{bemolle} is an isomorphism. Assume
also that there exists a vector field $\xi$ such that $i_\xi\omega=0$
and $\eta(\xi)=1$. Then, $M$ has odd dimension and $(M,\eta,\omega)$
is an almost cosymplectic manifold with Reeb vector field $\xi$.
\end{proposition}

By means of the isomorphism $\flat_{(\eta,\omega)}:\X(M)\to
\Omega^1(M)$ one can associate with every function $f\in
C^\infty(M)$ the vector field $\grad f\in\X(M)$, called \emph{gradient
vector field}, which is defined by
\[
 \grad f=\flat_{(\eta,\omega)}^{-1}(\d f).
\]
Equivalently one has (cf. \cite{albert89})
\begin{equation}\label{gradient}
 i_{\grad f}\,\omega=\d f-\xi(f)\eta,\qquad  \eta({\grad f})=\xi(f).
\end{equation}
If the manifold $M$ is cosymplectic, then it is easy to check that the
bracket defined on $C^\infty(M)$ by
\begin{equation}\label{poisson}
 \pb{f,g}=\omega(\grad f,\grad g),
\end{equation}
is a Poisson bracket. Thus, every cosymplectic manifold
$(M,\eta,\omega)$ has an induced Poisson structure. Its Poisson
bivector field is defined by
\begin{equation}\label{bivector}
 \pi_{(\eta,\omega)}(\alpha,\beta)=\omega(\flat_{(\eta,\omega)}^{-1}\alpha,\flat_{(\eta,\omega)}^{-1}\beta),
\end{equation}
for each $\alpha,\beta\in\Omega^1(M)$. Its symplectic leaves are
precisely the leaves of the integrable distribution $\ker \eta$. For
further details and properties of the gradient vector fields we
refer to \cite{cantrijn92}.

With every function $f\in C^\infty(M)$ one can also associate a
Hamiltonian vector field $\Ha_{f}$ according to
\[
\Ha_{f}=\flat_{(\eta,\omega)}^{-1}(\d
f-\xi(f)\eta)
\]
or equivalently
\begin{equation}\label{hamiltonian}
 i_{\Ha_{f}}\omega=\d f-\xi(f)\eta,\qquad  \eta({\Ha_{f}})=0.
\end{equation}
One can show (see \cite{albert89}) that the map $C^\infty(M)\to
\X(M)$ defined by $f\mapsto \Ha_{f}$ is a Lie algebra
anti-homomorphism with respect to the Poisson bracket
\eqref{poisson} and the commutator of vector fields, i.e.
\begin{equation}\label{antihomom}
 \Ha_{\pb{f,g}}=-[\Ha_{f},\Ha_{g}].
\end{equation}
From \eqref{gradient} and \eqref{hamiltonian} we get easily that
\begin{equation}\label{uguali}
 i_{\Ha_{f}}\omega=i_{\grad f}\,\omega
\end{equation}
and $\Ha_{f}=\grad f$ if and only if $\xi(f)=0$. Thus, the Poisson
structure of a cosymplectic manifold $(M,\eta,\omega)$ can be
equivalently written as
\begin{equation*}
 \pb{f,g}=\omega(\Ha_{f},\Ha_{g}).
\end{equation*}

 Any cosymplectic manifold $(M,\eta,\omega)$
of dimension $2n+1$ admits around any point local coordinates
$(t,q^i,p_i)$, $i=1,\ldots,n$, such that
\begin{equation}\label{darboux}
\omega=\sum_{i=1}^n \d q^i\wedge \d p_i,\qquad \eta=\d t,\qquad \xi=\frac{\partial}{\partial t}.
\end{equation}
This fact is usually referred to as \emph{Darboux theorem}.

Cosymplectic manifolds can be thought of as an odd-dimensional
counterpart of symplectic manifolds. In fact, on any cosymplectic
manifold $(M,\eta,\omega)$ the so-called horizontal distribution
$\ker \eta$ is integrable to a symplectic foliation of codimension
1. On the other hand, one has the following result on the product
$M\times\Rr$.
\begin{proposition}[\cite{deleon93}]\label{MxR}
Let $M$ be a manifold and $\eta,\omega$ two differential forms on
$M$ with degrees $1$ and $2$ respectively. Consider on $\tilde
M=M\times\Rr$ the differential $2$-form $\tilde\omega=p^*\omega +
p^*\eta\wedge \d t$, where $t$ is the coordinate function on $\Rr$
and $p:\tilde M\to M$ is the canonical projection. Then
\begin{itemize}
  \item[(a)] $(M,\eta,\omega)$ is a cosymplectic manifold if and
  only if $(\tilde M,\tilde\omega)$ is a symplectic manifold.
  \item[(b)] In such a case, $p$ is a Poisson morphism.
\end{itemize}
\end{proposition}

\begin{remark}
As in symplectic geometry, at any given point $p$ of a cosymplectic
manifold $M$ we can define, following \cite{deleon96}, a
cosymplectic orthogonal complement of a vector subspace $E\subset T_p M$ by
\begin{equation*}
E^\perp:=\{X\in T_p M\;|\; \eta(X)=0,\; i_X\omega\in E^o\},
\end{equation*}
where $E^o\subset T^*_p M$ is the annihilator of $E$. Then, we say
that a submanifold $C\subset M$  is a \emph{coisotropic submanifold}
of the cosymplectic manifold $(M,\eta,\omega)$ if
\begin{itemize}
  \item[(i)] the Reeb vector field $\xi$ is tangent to $C$;
  \item[(ii)] $(TC)^\perp\subset TC$.
\end{itemize}
The notion of a coisotropic submanifold $C$ of a cosymplectic
manifold $M$ can be related to the usual notion of a coisotropic
submanifold of a symplectic manifold in the following way. If
$(M,\eta,\omega)$ is cosymplectic, then the manifold $\tilde
M=M\times\Rr$ has a symplectic structure defined by Proposition
\ref{MxR}. One then can prove that $C\subset M$ is coisotropic if
and only if one has $(T(C\times\Rr))^\perp\subset TC\oplus \{0\}$,
where $(T(C\times\Rr))^\perp$ denotes the symplectic orthogonal complement. It
follows that $C\times\Rr$ is a special kind of coisotropic
submanifold of $M\times\Rr$ in the symplectic sense.
\end{remark}

\begin{definition}
A smooth map $\psi:M\to N$ between two cosymplectic manifolds
$(M,\eta,\omega)$ and $(N,\eta',\omega')$ is said to be
\emph{cosymplectic} if $\psi^*\eta'=\eta$ and
$\psi^*\omega'=\omega$. In such a case, the Reeb vector field $\xi$
of $M$ is $\psi$-projectable and its projection is the Reeb vector
field $\xi^\prime$ of $N$, that is,
\[
 T\psi\circ\xi=\xi'\circ\psi.
\]
\end{definition}

As in the symplectic case, a cosymplectic map is not, in general, a
Poisson map. In a cosymplectic manifold $(M,\eta,\omega)$ one can
use the vector bundle isomorphism $\flat_{(\eta,\omega)}:TM\to T^*M$
to pull back the canonical symplectic 2-form $\Omega_M$ of the
cotangent bundle $T^*M$ to the tangent bundle $TM$ by setting
\[
 \Omega_o:=\flat_{(\eta,\omega)}^*\Omega_M.
\]
Therefore the tangent bundle of a cosymplectic manifold is a
symplectic manifold. Actually, the existence of the vector bundle
isomorphism $\flat_{(\eta,\omega)}:TM\to T^*M$ means that the tangent
bundle of a cosymplectic manifold has a Liouville structure (also
known as special symplectic structure) in the sense of Tulczyjew
\cite{tulczyjew09}. In \cite{cantrijn92} the authors obtain an
explicit expression for the  symplectic structure $\Omega_o$ on the
tangent bundle in terms of the cosymplectic structure
$(\eta,\omega)$, by using the notion of tangent derivation, again
due to Tulczyjew (see e.g. \cite{tulczyjew76}).

\subsection{Reduction of cosymplectic structures}

An action $\phi:G\times M\to M$ of a Lie group $G$ on a cosymplectic
manifold $(M,\eta,\omega)$ is said to be \emph{cosymplectic} if
$\phi_g:M\to M$ is a cosymplectic map for any $g\in G$.  In such a
case, a smooth map $J:M\to\g^*$ from $M$ to the dual space $\g^*$ of
the Lie algebra $\g$ of $G$ is said to be a \emph{momentum map} if
the infinitesimal generator $a_M\in\X(M)$ of the action associated
with any $a\in\g$ is the Hamiltonian vector field of the function
$J_a:M\to\Rr$ defined by the natural pointwise pairing. Moreover,
$J$ is said to be \emph{$\Ad^*$-equivariant} if it is equivariant
with respect to the action $\phi$ and to the coadjoint action
$\Ad^*:G\times\g^*\to\g^*$, that is
\[
 J(\phi_g(x))=\Ad^*_{g^{-1}}(J(x)),\qquad\mbox{for any }x\in M.
\]

An element $\nu\in\g^*$ is said to be a \emph{weakly regular value}
of $J$ if $J^{-1}(\nu)$ is a closed submanifold of $M$ and for each
$x\in J^{-1}(\nu)$ the tangent space $T_x(J^{-1}(\nu))$ coincides
with the kernel of $T_xJ$. For example, if $\nu\in\g^*$ is a regular
value of $J$ (i.e. if $T_xJ$ has full rank for every $x\in
J^{-1}(\nu)$), then it is a weakly regular value. Note that, if
$\phi:G\times M\to M$ is a free and proper action, and $J:M\to\g^*$
is a momentum map for $\phi$, then one can prove (cf. \cite{arms81})
that $J$ is a submersion, i.e. all elements $\nu\in\g^*$ are regular
values.

Let $\nu\in\g^*$ be a weakly regular value of an $\Ad^*$-equivariant
momentum map $J:M\to\g^*$ for an action $\phi:G\times M\to M$. If
$G_\nu$ denotes the isotropy group of $\nu$ with respect to the
coadjoint action, i.e.
\[
 G_\nu=\{g\in G:\,\Ad^*_{g}\nu=\nu\},
\]
then $\phi$ induces an action
\[
 \phi:G_\nu\times J^{-1}(\nu)\to J^{-1}(\nu)
\]
of $G_\nu$ on the submanifold $J^{-1}(\nu)$. Following Libermann and
Marle \cite{libermann87} we will say that this action is
\emph{simple} (Albert calls such an action ``quotientant'') if the
orbit space $M_\nu=J^{-1}(\nu)/G_\nu$ admits a smooth manifold
structure and the canonical projection $\pi_\nu:J^{-1}(\nu)\to
M_\nu$ is a surjective submersion. A case in which this happens is
when the action is free and proper. If the action $\phi:G\times M\to
M$ is  free and proper, then the induced action of $G_\nu$ on
$J^{-1}(\nu)$ is again free and proper, for each $\nu\in\g^*$.

One has the following reduction theorem that yields an analogue of
the well-known reduction theorem for symplectic manifolds
\cite{marsden74}.

\begin{theorem}[\bf{Cosymplectic reduction Theorem}, \cite{albert89}]\label{CosymplecticReduction}
Let $\phi:G\times M\to M$ be a  cosymplectic action of a Lie group
$G$ on a cosymplectic manifold $(M,\eta,\omega)$. Suppose that
$J:M\to\g^*$ is an $\Ad^*$-equivariant momentum map associated with
$\phi$ such that
\begin{equation}\label{ReebHypothesis}
 \xi(J_a)=0,\qquad\mbox{for any }a\in\g,
\end{equation}
where $\xi$ is the Reeb vector field of $M$. Let $\nu\in\g^*$ be a
weakly regular value of $J$ such that the induced action of $G_\nu$
on $J^{-1}(\nu)$ is simple. Then, $M_\nu=J^{-1}(\nu)/G_\nu$ is a
cosymplectic manifold with cosymplectic structure
$(\eta_\nu,\omega_\nu)$ characterized by
\begin{equation}\label{ReducedCosymplecticStruct}
 \pi_\nu^*\eta_\nu=i_\nu^*\eta, \qquad
 \pi_\nu^*\omega_\nu=i_\nu^*\omega,
\end{equation}
where $\pi_\nu:J^{-1}(\nu)\to M_\nu$ is the canonical projection and
$i_\nu:J^{-1}(\nu)\hookrightarrow M$ is the canonical inclusion.

Moreover, the restriction $\xi_{|J^{-1}(\nu)}$ of $\xi$ is tangent
to $J^{-1}(\nu)$ and $\pi_\nu$-projectable onto the Reeb vector field
$\xi_\nu$ of $M_\nu$.
\end{theorem}

Since  each cosymplectic manifold is a Poisson manifold, the above
reduction problem can be enclosed in the more general case of the
reduction of Poisson structures, studied by Marsden-Ratiu
\cite{marsden86}. The supplementary information added by Albert
consists in the cosymplectic nature of the reduced manifold.
Furthermore, as it has been remarked in \cite{lacirasella12}, one
can prove that the Poisson structure induced by
$(\eta_\nu,\omega_\nu)$ coincides with the reduced Poisson structure
deduced by the Marsden-Ratiu reduction theorem for Poisson
manifolds.

In \cite{deleon93} de Le{\'o}n  and Saralegi showed that the Albert
reduction theorem can be obtained in a different way, by first
extending the phase space $M$ according to Proposition \ref{MxR} and
then applying the  Marsden-Weinstein reduction procedure for
symplectic manifolds \cite{marsden74}. This allows them to study
cosymplectic singular reduction, that is the reduction for singular
values of the momentum map, by applying the results of Sjamaar and
Lerman \cite{sjamaar91} for singular symplectic reduction. In this
way, the authors of \cite{deleon93} are able to obtain a stratified
cosymplectic space as a result of the reduction of a cosymplectic
manifold with a cosymplectic action of a compact group and a
momentum map.

\subsection{Universal model for cosymplectic manifolds}

The following more general kind of reduction of cosymplectic
manifolds was considered by de Le{\'o}n and Tuynman \cite{deleon96}.

\begin{definition}
Let $(M,\eta,\omega)$ be a cosymplectic manifold and let $C$ be a
submanifold of  $M$. Suppose furthermore that the following three
conditions are satisfied:
\begin{itemize}
  \item [(i)]  the Reeb vector field $\xi$ is tangent to the submanifold $C$;
  \item [(ii)]  the characteristic distribution ${\mathcal F}:=\ker \eta|_{C}\cap\ker \omega|_{C}$ has constant rank and hence it is a regular foliation on $C$;
  \item [(iii)] the foliation ${\mathcal F}$ is simple, i.e. the space of leaves $M_{r}:=C/{\mathcal F}$ has a structure of a manifold and the canonical projection $\pi:C\to M_{r}$ is a submersion.
\end{itemize}
\end{definition}
Under the above hypotheses, it can be shown that  there exist unique
closed forms $\eta_{r}\in\Omega^{1}(M_{r})$ and
$\omega_{r}\in\Omega^{2}(M_{r})$  such that
$(M_{r},\eta_{r},\omega_{r})$ is a cosymplectic manifold, the
pullbacks of $\pi^{*}\eta_{r}$ and $\pi^{*}\omega_{r}$  coincide
with the restrictions $\eta|_{C}$ and $\omega|_{C}$, respectively,
and the Reeb vector field $\xi$ of $M$ projects onto the Reeb vector
field $\xi_{r}$ of the cosymplectic manifold $M_{r}$.

In these circumstances we will say that $M_{r}$ is the
\emph{reduction} of $M$ by $C$.

Some properties of the cosymplectic reduction are collected in the
following proposition.
\begin{proposition}[\cite{deleon96}]\label{reductionC}
Let $(M_{r},\eta_{r},\omega_{r})$  be the reduction of a
cosymplectic manifold $(M,\eta,\omega)$ by a submanifold $C\subset
M$. Then, $\eta$ is exact implies that $\eta_r$ is also exact. If
the Reeb vector field $\xi$ is complete and $C$ is closed in $M$,
then $\xi_r$ is also complete.
\end{proposition}

One could wonder if  the local models for cosymplectic manifolds,
which are given by the Darboux coordinates \eqref{darboux} on
$\Rr^{2n+1}$, could be general universal models, in the sense that
all cosymplectic manifolds can be obtained from an $\Rr^{2N+1}$ by
reduction with respect to a submanifold. Indeed in the Darboux local
model the 1-form is exact and the Reeb vector field is complete.
Since for a general cosymplectic manifold the 1-form $\eta$ need not be
exact, Proposition \ref{reductionC} shows that local models cannot be
universal.  Moreover, since the Reeb vector field of a general
cosymplectic manifold need not be complete, it also follows from
Proposition~\ref{reductionC} that we cannot always take the submanifold $C$ to
be closed. This is in sharp contrast with the symplectic case,
studied in \cite{gotay89}, in which the local model $\Rr^{2N}$ is at
the same time the universal model, and in which the reduction can
always be done by a closed submanifold. The universal model for
cosymplectic manifolds found by de Le{\'o}n and Tuynman turns out to
be slightly more complicated.

Let us say that a manifold is of \emph{cohomologically finite type}
if the de Rham cohomology $H^k(M)$ is a finitely generated vector
space for all $k$. For any 1-form $\eta$ on a manifold $M$ we define
$\rank_{\Qq}(\eta)$ by
\[
\rank_{\Qq}(\eta)=\dim_{\Qq}(\Qq \cdot \mbox{Periods}(\eta)),
\]
where
\[\mbox{Periods}(\eta)=\left\{\int_z\eta\;|\; z \mbox{ a chain in
}M\right\}\subset\Rr.
\]
Note that for a  manifold $M$ of cohomologically finite type the
$\Qq$-dimension of $\Qq \cdot \mbox{Periods}(\eta)$ is always
finite. We can now state the following theorem which gives the
universal model for cosymplectic manifolds.

\begin{theorem}[\cite{deleon96}]\label{universal}
Let $(M,\eta,\omega)$ be a cosymplectic manifold of cohomologically
finite type. Then there exist integers $N$ and $k$ and real numbers
$\mu_1,\ldots,\mu_k$ that are independent over $\Qq$ such that $M$
is the reduction of the cosymplectic manifold
$(M_u,\eta_u,\omega_u)$ by some submanifold $C\subset M_u$, where
\begin{equation*}
M_u=\Rr\times T^*(\T^k\times \Rr^N), \qquad \eta_u=\d \theta_{\T^k\times \Rr^N},\qquad \omega_u=\d t+\sum_{i=1}^k\mu_i\d \varphi_i,
\end{equation*}
where $\theta_{\T^k\times \Rr^N}$ is the Liouville 1-form,
$\varphi_i$ are the angle coordinates on the torus, $t$ is the
canonical coordinate on $\Rr$, and $k=\rank_{\Qq}(\eta)$.
\end{theorem}

De Le{\'o}n and Tuynman also proved  a corresponding equivariant
version of the above theorem, for which we refer to \cite{deleon96}.

\subsection{Canonical homology}
The canonical homology was introduced by Koszul in~\cite{koszul85} for any
Poisson manifold. As
every cosymplectic manifold can be considered as a Poisson manifold as well,
it is interesting to study the properties of the canonical homology for them.

Let $M$ be a Poisson manifold and $\pi$ the Poisson bivector. Then
we can define the map $\delta\colon \Omega^p\left( M \right)\to
\Omega^{p-1}\left( M \right)$ by $\delta = \left[ i_\pi, d \right]$,
where $i_\pi$ is the operator of contraction with the bivector
$\pi$. It is easy to check that $\delta^2 = 0$. Therefore $\left(
\Omega^p\left( M \right), \delta \right)$ is a complex, whose
homology $H_p^{can}\left( M \right)$ is called \emph{canonical
homology} of $M$.

In the case $M^{2n}$ is a compact symplectic manifold it was shown
in~\cite{brylinski88} that $H_p^{can}\left( M \right)\cong
H^{2n-p}\left( M \right)$ for $0\le p\le n$.

It was discovered in~\cite{fernandez-ibanez98}  that a similar
result for a compact cosymplectic manifold $M^{2n+1}$ is false by
constructing an explicit counterexample of dimension $5$.
Nevertheless, in~\cite{fernandez-ibanez98} Fern{\'a}ndez et al. were able to construct a spectral
sequence whose first page is given by $H_p^{can}\left( M \right)$
and that converges  to $H^{2n+1-p}\left( M \right)$. In particular,
for any compact cosymplectic manifold $M$
$$
\dim H_p^{can}\left( M \right) \ge b_{2n+1-p}\left( M \right), \ \ 0\le p\le
2n+1.
$$
Further topological properties of cosymplectic manifolds will be
reviewed in the almost coK\"{a}hler context in Sections
\ref{topology} and \ref{coeffective}.


\section{CoK\"{a}hler structures} \label{cokahler}

Historically  coK\"{a}hler manifolds were defined  by Blair in
\cite{blair67} in the context of quasi-Sasakian manifolds and,
independently, by Ogiue in \cite{ogiue68} as normal almost contact
metric manifolds such that the fundamental $2$-form $\Phi$ and the
$1$-form $\eta$ are closed.  Actually the notion was used, even if
implicitly, for the first time in an article of Okumura
(\cite{okumura65}), where he studied the hypersurfaces  of a
K\"{a}hler manifold of constant holomorphic sectional curvature. Okumura
studied the conditions for which such a hypersurface carries a
cosymplectic structure, but he did not realize that in fact he was
dealing with a general geometric structure richer than the
cosymplectic one.

It should be remarked that Ogiue used the word ``cocomplex'',
whereas Blair the word  ``cosymplectic''. The latter has been used
in many subsequent papers and it is now widely spread across the
literature. However, since it  could be confusing, as the same term
was used by Libermann for a different (although related) geometric
structure, here we prefer the word ``coK\"{a}hler''. This name was
firstly used by Bejancu and Smaranda (\cite{bejancu83}) and, later
on, by Janssens, Vanecke (\cite{janssens-vanhecke81}) and  Marrero
(\cite{marrero90}). More recently the paper of Li (\cite{li08}) and
some other subsequent papers have adopted this terminology.

\subsection{Basic definitions and properties}
An \emph{almost contact structure} on a manifold $M$ is given by a
triple $(\phi,\xi,\eta)$, where $\xi$  is a  vector field, $\eta$ is
a $1$-form and $\phi$ is
 a $(1,1)$-tensor field related by the following conditions
\begin{equation*}
\phi^2 = -I + \eta\otimes\xi, \ \ \ \eta(\xi)=1,
\end{equation*}
where $I$ denotes the identity mapping. From the definition it
follows that $M$ has necessarily odd dimension, say $2n+1$. Further
consequences are that $\phi\xi=0$ and $\eta\circ\phi =0$ (see
\cite{blairbook2010} for the proof and any detail). This notion was
introduced in 1961 by Sasaki (\cite{sasaki61}). In the same year
Sasaki and Hakateyama (\cite{sasaki-hatakeyama61}) introduced the
odd-dimensional counterpart of the notion of complex manifold. Their
idea was to consider the product $M\times \mathbb{R}$ of an almost
contact manifold with the real line and to define on it a tensor
field $J$ of type $(1,1)$ by
\begin{equation}\label{complex}
J\left(X,f\frac{\d}{\d t}\right)=\left(\phi X - f\xi, \eta(X)\frac{\d}{\d t}\right)
\end{equation}
for any $X\in\Gamma(TM)$ and $f\in C^{\infty}(M\times\mathbb{R})$. One easily checks that $J^{2}=-I$,
 giving an almost complex structure on $M\times \mathbb{R}$. Then one says that, by definition, the almost contact structure $(\phi,\xi,\eta)$
 is \emph{normal} if $J$ is a complex structure. Some long but straightforward computations of the Nijenhuis torsion of $J$ yield, for any
 $X,Y\in\Gamma(TM)$,
\begin{gather*}
[J,J]\left((X,0),(Y,0)\right)=\left(N^{(1)}(X,Y),N^{(2)}(X,Y)\frac{\d}{\d t}\right)\\
[J,J]\left((X,0),\left(0,\frac{\d}{\d
t}\right)\right)=\left(N^{(3)}(X),N^{(4)}(X)\frac{\d}{\d t}\right),
\end{gather*}
where we have put
\begin{gather}
N^{(1)}:=[\phi,\phi]+2\d\eta\otimes\xi \label{n1} \\
N^{(2)}(X,Y):=({\mathcal L}_{\phi X}\eta)(Y)-({\mathcal L}_{\phi Y}\eta)(X) \label{n2}  \\
N^{(3)}:={\mathcal L}_{\xi}\phi \label{n3}  \\
N^{(4)}:={\mathcal L}_{\xi}\eta.  \label{n4}
\end{gather}
Therefore $(\phi,\xi,\eta)$ is normal if and only if all the above fundamental tensors $N^{(1)}$, $N^{(2)}$, $N^{(3)}$, $N^{(4)}$ vanish identically.
However, it was proved that the vanishing of $N^{(1)}$ implies the vanishing of the remaining tensors. Thus one has the following theorem.

\begin{theorem}[\cite{sasaki-hatakeyama61}]
An almost contact manifold $(M,\phi,\xi,\eta)$ is normal if and only if the tensor field $N^{(1)}$ vanishes identically.
\end{theorem}

A Riemannian metric $g$ on $M$ is said to be \emph{compatible} with the almost contact structure if it satisfies
\begin{equation}\label{compatible1}
g(\phi X, \phi Y) = g(X,Y) - \eta(X)\eta(Y)
\end{equation}
for any $X,Y\in\Gamma(TM)$. Such a metric always exists, provided that $M$ is paracompact. If we fix one, say $g$, we call the geometric structure $(\phi,\xi,\eta,g)$ an \emph{almost contact metric structure}.
Note that, taking $Y=\xi$ in \eqref{compatible1} we get
\begin{equation}\label{metric1}
\eta(X)=g(X,\xi).
\end{equation}
Moreover, \eqref{compatible1} implies that $g(\phi X, Y)=-g(X,\phi Y)$ for any $X,Y\in\Gamma(TM)$, so that the bilinear form $\Phi$ defined by
\begin{equation*}
\Phi(X,Y):=g(X,\phi Y)
\end{equation*}
is in fact a $2$-form, called  \emph{fundamental $2$-form} (or \emph{Sasaki form}) of the almost contact metric manifold $(M,\phi,\xi,\eta,g)$. Since $\phi\xi=0$ one easily gets
\begin{equation}\label{sasakiform1}
i_{\xi}\Phi=0
\end{equation}
Due to \eqref{metric1}, we can decompose the tangent bundle of an almost contact metric manifold as the orthogonal sum of the codimension $1$ distribution
${\mathcal D} = \ker(\eta)$ and the $1$-dimensional distribution $\mathcal V$
defined by the vector field $\xi$.

There are many remarkable classes of almost contact metric manifolds
(for a full classification see \cite{chinea-gonzalez90}). One of
them is given by \emph{quasi-Sasakian manifolds}, that is normal
almost contact metric manifolds whose fundamental $2$-form is
closed. This class of almost contact metric manifolds, which
resembles the class of K\"{a}hler manifolds in even dimension, was
defined by Blair in his Ph.D. thesis (see \cite{blair67}) and then
studied by several authors. Among Blair's results there is the fact that
there are no quasi-Sasakian manifolds of even rank. Recall that  an
almost contact manifold $(M,\phi,\xi,\eta)$ of dimension $2n+1$ is
said to be of rank $2p+1$ if, for some $p\leq n$,  one has that
$\eta\wedge(\d\eta)^p \neq 0$ at any point of $M$ and
$(\d\eta)^{p+1}\equiv 0$, whereas one says that $M$ has rank $2p$ if
$(\d\eta)^p \neq 0$ and $\eta\wedge (\d\eta)^{p} \equiv 0$ on $M$. If the
rank is maximal,  $\eta$ becomes a contact form. We now
examine the minimal rank case.

\begin{definition}
An \emph{almost coK\"{a}hler manifold} is an almost contact metric manifold $(M,\phi,\xi,\eta,g)$ such that the fundamental $2$-form $\Phi$ and the
$1$-form $\eta$ are closed. If in addition the almost contact structure is normal, we say that $M$ is a \emph{coK\"{a}hler manifold}.
\end{definition}

Any almost coK\"{a}hler manifold is canonically a cosymplectic
manifold, with cosymplectic structure determined by the $1$-form
$\eta$, the $2$-form $\Phi$, and $\xi$ as Reeb vector field. This can be proved by considering a
\emph{$\phi$-basis}, that is a local orthonormal basis
$\left\{X_{1},\ldots,X_{n},Y_{1},\ldots,Y_{n},\xi\right\}$, where
for each $i\in\left\{1,\ldots,n\right\}$, $Y_i=\phi X_i$. Such a
basis exists in every almost contact metric manifold (see e.g. page
44 of \cite{blairbook2010}). Then
\begin{equation*}
\eta\wedge\Phi^n (\xi,X_{1},\ldots,X_{n},Y_{1},\ldots,Y_{n}) \neq 0.
\end{equation*}
Conversely, any cosymplectic manifold carries an associated almost
coK\"{a}hler structure, as shown in the following theorem.

\begin{theorem}[cf. \cite{blair76,deleon89}]\label{metrizability}
Let $(M,\eta,\omega)$  be an almost cosymplectic manifold of
dimension $2n+1$ with Reeb vector field $\xi$. Then, there exists an almost contact metric
structure $(\phi,\xi,\eta,g)$ on $M$ (with the same $\eta$ and
$\xi$), whose fundamental $2$-form $\Phi$ coincides with $\omega$.
\end{theorem}
\begin{proof}
    Let $\xi$ be the Reeb vector field of $\left( M,\eta, \omega
    \right)$ (see page~\pageref{reeb}). It satisfies the equalities
 $$\eta(\xi)=1, \qquad i_{\xi}\omega=0.$$
 For every non-zero $X\in \Gamma(\ker (\eta))$, we have $i_X \omega = i_X \omega + \eta\left( X
 \right)\eta = \flat_{\left( \eta, \omega \right)}\left( X \right)$, which is
 non-zero since $\flat_{\left( \eta, \omega \right)}$ is an isomorphism.
 This shows that the restriction of $\omega$ to $ {\mathcal D}:= \ker (\eta) $ is a
 non-degenerate quadratic form. Denote by $\Rr\xi$ the distribution generated by
 $\xi$.
Then there exists a vector bundle  endomorphism $\phi_{{\mathcal
D}}:{\mathcal D}\to {\mathcal D}$ and a metric $g_{{\mathcal D}}$ on
${\mathcal D}$ such that for any vector fields $X,Y\in
\Gamma({\mathcal D})$,
$$g_{{\mathcal D}}(X,\phi_{{\mathcal D}} Y)=\omega(X,Y),$$
and $\phi_{{\mathcal D}}^{2}=-I_{{\mathcal D}}$, where $I_{{\mathcal
D}}$ denotes the identity map on ${\mathcal D}$  (see for
example~\cite{hummel97}). Now, we define a metric $g$ on $M$ by
$$g(X, Y)=g_{{\mathcal D}}(X, Y), \quad g(X, \xi)=0, \quad g(\xi, \xi)=1,$$
for any $X,Y\in \Gamma({\mathcal D})$. Next, we define a vector
bundle endomorphism  $\phi:TM\to TM$ by
$$ \phi(X)=\phi_{{\mathcal D}}(X),\quad \phi(\xi)=0,$$
for any  $X\in \Gamma({\mathcal D})$. Then, it is easy to check that
$(M,\phi,\xi,\eta,g)$ is an almost contact metric structure and
$$g(X,\phi Y)=\omega(X,Y).$$
\end{proof}

\begin{corollary}
Any cosymplectic manifold $(M,\eta,\omega)$ admits an almost
coK\"{a}hler structure $(\phi,\xi,\eta,g)$ with fundamental 2-form
$\Phi=\omega$.
\end{corollary}

CoK\"{a}hler manifolds were studied for the first time in \cite{blair67} and \cite{blair-goldberg67} in the context of quasi-Sasakian manifolds, since they are quasi-Sasakian manifolds of rank $1$.

We collect some basic properties of almost coK\"{a}hler
manifolds. The first is that, due to the closedness of $\eta$,  one
has that $\mathcal D$ is integrable and, moreover, the Reeb vector
field is an infinitesimal automorphism with respect to $\mathcal D$,
that is $[\xi,X]\in\Gamma(\mathcal D)$ for any $X\in\Gamma(\mathcal
D)$.

Notice that, by \eqref{sasakiform1} and the Cartan formula for the
Lie derivative, one obtains
\begin{equation}\label{sasakiform2}
{\mathcal L}_{\xi}\Phi=0.
\end{equation}

The next theorem shows some properties which almost coK\"{a}hler manifolds share with contact metric manifolds (cf. \cite[Theorem 6.2]{blairbook2010}):

\begin{theorem}\label{almostk1}
For an almost coK\"{a}hler structure $(\phi,\xi,\eta,g)$ the  tensor fields $N^{(2)}$ and $N^{(4)}$ vanish. Moreover, $N^{(1)}$ coincides with the
Nijenhuis tensor of $\phi$ and $N^{(3)}$ vanishes if and only if $\xi$ is a Killing vector field.
\end{theorem}
\begin{proof}
The claim straightforwardly follows once one observes that $N^{(2)}$ and $N^{(4)}$ can be written as
\begin{gather*}
N^{(2)}(X,Y)=2\d\eta(\phi X,Y)-2\d\eta(X,\phi  Y), \ \ \
N^{(4)}(X)=2\d\eta(\xi,X).
\end{gather*}
For the last part of the theorem, notice that by \eqref{sasakiform2} one has
\begin{equation*}
0=({\mathcal L}_{\xi}\Phi)(X,Y)=({\mathcal L}_{\xi}g)(X,\phi Y)+g(X,({\mathcal L}_{\xi}\phi) Y),
\end{equation*}
from which, by using the vanishing of $N^{(4)}$, the assertion follows.
\end{proof}

\begin{corollary}
In any almost coK\"{a}hler manifold the integral curves of the Reeb vector field are geodesics.
\end{corollary}
\begin{proof}
The vanishing of $N^{(4)}$ yields, for any $X\in\Gamma(TM)$,
\begin{align*}
0&=\xi(\eta(X))-\eta([\xi,X])\\
&=\xi(g(X,\xi))-g(\nabla_{\xi}X,\xi)+g(\nabla_{X}\xi,\xi)\\
&=g(X,\nabla_{\xi}\xi),
\end{align*}
since $g(\nabla_{X}\xi,\xi)=X(g(\xi,\xi))-g(\xi,\nabla_{X}\xi)$, from which it follows that $2g(\nabla_{X}\xi,\xi)=0$.
\end{proof}

Another general property of almost coK\"{a}hler manifolds concerns the harmonicity of the forms $\eta$ and $\Phi$:

\begin{theorem}[\cite{goldberg-yano69}]\label{harmonicity}
The $1$-form $\eta$ and the $2$-form $\Phi$ of any almost coK\"{a}hler manifold are harmonic.
\end{theorem}

Now recall the following general formula, which holds for any almost contact metric manifold (\cite[Lemma 6.1]{blairbook2010}),
\begin{align*}
2g((\nabla_{X}\phi)Y,Z)&=3\d\Phi(X,\phi Y,\phi Z) - 3\d\Phi(X, Y, Z) + g(N^{(1)}(Y,Z),\phi X)\\
&\quad + N^{(2)}(Y,Z)\eta(X) + 2\d\eta(\phi Y,X)\eta(Z) - 2\d\eta(\phi Z,X)\eta(Y).
\end{align*}

Therefore we have the following result for almost coK\"{a}hler manifolds.

\begin{corollary}
In any almost coK\"{a}hler manifold one has
\begin{equation}\label{nablaphi1}
g((\nabla_{X}\phi)Y,Z)=  \frac{1}{2}g([\phi,\phi](Y,Z),\phi X).
\end{equation}
\end{corollary}

From \eqref{nablaphi1} it follows in particular that
\begin{equation}\label{nablaphixi}
\nabla_{\xi}\phi=0,
\end{equation}
which was already proved in another way by Olszak (\cite{olszak81}). Another easy consequence of \eqref{nablaphi1} is that if $M$ is coK\"{a}hler then
$\phi$ is parallel. However also the converse holds. In fact we have the following more general result.

\begin{theorem}[\cite{blair67}]\label{cokahlercondition1}
An almost contact metric manifold $(M,\phi,\xi,\eta,g)$ is coK\"{a}hler if and only if $\nabla\phi=0$ or, equivalently, $\nabla\Phi=0$.
\end{theorem}

On an almost coK\"{a}hler manifold it is often convenient to use the operator
\begin{equation*}
h:=\frac{1}{2}{\mathcal L}_{\xi}\phi=\frac{1}{2}N^{(3)}.
\end{equation*}
The main properties of this tensor field are described in the following proposition, which can be proved straightforwardly (see also \cite{endo1994-1}).

\begin{proposition}
Let $(M,\phi,\xi,\eta,g)$ be an almost coK\"{a}hler manifold. Then the operator $h$ satisfies the following properties:
\begin{enumerate}
\item[(i)] $h$ is a symmetric operator with respect to $g$;
\item[(ii)] $h$ anti-commutes with $\phi$;
\item[(iii)] $\nabla\xi = - \phi h $.
\end{enumerate}
\end{proposition}

Then we can state the following important result.

\begin{theorem}\label{killing}
Let $(M,\phi,\xi,\eta,g)$ be an almost coK\"{a}hler manifold. The following statements are equivalent:
\begin{enumerate}
\item[(a)] the operator $h$ vanishes identically;
\item[(b)] the Reeb vector field is parallel;
\item[(c)] the Reeb vector field is Killing;
\item[(d)] the foliation $\mathcal D$ defined by the kernel of $\eta$ is totally geodesic;
\item[(e)] $M$ is locally the Riemannian product of an almost K\"{a}hler manifold with the real line;
\item[(f)] $\nabla \eta = 0$.
\end{enumerate}
\end{theorem}

\begin{remark}\label{3D-Killing}
Since in any coK\"{a}hler manifold, as $N^{(1)}=0$, also the tensor field $N^{(3)}$ vanishes, by Theorem \ref{killing} we have that $\xi$ is Killing. The
converse, in general, does not hold. However, it was proved that in dimension
$3$ any among the above six conditions is equivalent to requiring that the
manifold is coK\"{a}hler (\cite{goldberg-yano69}).
\end{remark}

Concerning the operator $h$, we state the following general identity, due to Olszak (\cite{olszak81}), which relates $h$ with the covariant derivative of
$\phi$ in any almost coK\"{a}hler manifold
\begin{equation}\label{identity1}
(\nabla_{\phi X}\phi)\phi Y + (\nabla_{X}\phi)Y -\eta(Y)hX =0.
\end{equation}
Notice that, from \eqref{identity1} it follows that $(\nabla_{\phi X}\eta)(\phi Y)=-(\nabla_{ X}\eta)(Y)$.

\subsection{Examples of coK\"{a}hler manifolds}
\ The standard example of (almost) coK\"{a}hler manifold is given by
the product of an (almost) K\"{a}hler manifold $(N,J,G)$ with
$\mathbb{R}$ (or $S^1$). If $t$ denotes the global coordinate on
$\mathbb R$, any vector field of $M$ can be written as $\left(X,f
\frac{\d}{\d t}\right)$ for some smooth function $f$ on the product.
Let us  define a tensor field $\phi$ of type $(1,1)$ by setting
$\phi\left(X,f \frac{\d}{\d t}\right):=\left(JX, 0\right)$, a vector
field $\xi:=\frac{\d}{\d t}$, a $1$-form $\eta:=\d t$ and a
Riemannian metric $g$ as the product metric $G+\d t^2$. Then one can
check that $(\phi,\xi,\eta,g)$ is an almost coK\"{a}hler structure
such that $\xi$ is Killing, and it is coK\"{a}hler  if and only if
$N$ is K\"{a}hler.


On the other hand by Theorem \ref{killing} every coK\"{a}hler  manifold is, at least locally, the Riemannian product of a K\"{a}hler manifold with the real
line.

Therefore  two natural problems rise:
\medskip

\textbf{Question 1} - \emph{Are there examples of almost coK\"{a}hler manifolds which are not, even locally,  the Riemannian product of an almost
K\"{a}hler manifold with the real line?}
\medskip

\textbf{Question 2} - \emph{Are there examples of coK\"{a}hler manifolds which are not the global product of a K\"{a}hler manifold with $\mathbb{R}$ or
$S^1$?}
\medskip

Clearly, in view of Theorem \ref{killing}, Question 1 is equivalent to the problem of finding examples of proper  almost coK\"{a}hler manifolds with
non-Killing Reeb vector field. The problem was solved by Olszak
(\cite{olszak81}), who found an example of an  almost coK\"{a}hler structure with non-parallel
Reeb vector field in any odd dimension. Namely, let $\mathfrak{g}$ be the $(2n+1)$-dimensional real solvable Lie algebra with basis
$\{E_{0},E_{1},\ldots,E_{2n}\}$ and non-zero Lie brackets
\begin{gather*}
[E_{0}, E_{i}] = -a_{i}E_{i} - a_{i+n}E_{i+n}, \ \ \ [E_{0}, E_{i+n}] = -a_{i+n}E_{i} + a_{i}E_{i+n},
\end{gather*}
where $a_{1},\ldots,a_{2n}$ are real numbers such that $a_{1}^2+\cdots+a_{2n}^2>0$. Let $G$ be a Lie group whose Lie algebra is $\mathfrak{g}$ and extend
$\{E_{0},E_{1},\ldots,E_{2n}\}$ to left invariant vector fields on $G$. We define on $G$ a Riemannian metric $g$ by imposing that
$\{E_{0},E_{1},\ldots,E_{2n}\}$ is orthonormal. Then we define a $1$-form $\eta$ and a tensor field $\phi$ by setting $\eta(E_j)=\delta_{0j}$ for each
$j\in\left\{0,1,\ldots,2n\right\}$ and $\phi E_{0}=0$, $\phi E_i = E_{i+n}$,  $\phi E_{i+n} = -E_{i}$  for each $i\in\left\{1,\ldots,n\right\}$. Then
putting $\xi:=E_0$, one can prove that $(\phi,\xi,\eta,g)$ is an almost coK\"{a}hler structure such that $\nabla\xi\neq 0$.

With regard to Question 2, it was positively solved in dimension $3$
by Chinea, de Le\'{o}n, Marrero in 1993
(\cite{chinea-deleon-marrero93}). We now describe such an example
explicitly.
Firstly let us consider the following general
construction, due to Fujimoto and Mut\={o} (\cite{fujimoto74}).
Let $(N,J,G)$ be a compact
K\"{a}hler manifold and $f:N\longrightarrow N$ be a Hermitian
isometry, that is an isometry of the Riemannian manifold $(N,g)$
such that $f_\ast \circ J = J \circ f_\ast$. Define an action
$\rho:\mathbb{Z}\times (N\times\mathbb{R})\longrightarrow
N\times\mathbb{R}$ of $\mathbb{Z}$ on $N\times\mathbb{R}$ by
$\rho(k,(x,t))=(f^{k}(x),t+k)$. This action is free and properly
discontinuous, so that the corresponding orbit space $M_f$ is a
compact smooth manifold. Then the coK\"{a}hler structure induced on
the product manifold $N\times\mathbb{R}$ descends to a coK\"{a}hler
structure on $M_{f}:=(N\times\mathbb{R})/\mathbb{Z}$. Now Chinea, de
Le\'{o}n, Marrero restrict to the case when $N$ is the
$2$-dimensional torus $\mathbb{T}^2$ with its standard K\"{a}hler
structure and $f$ is the Hermitian isometry induced by the isometry
of $\mathbb{R}^2$ $f(x,y)=(y,-x)$. They note that the manifold
$M_{f}:=(\mathbb{T}^2\times\mathbb{R})/\mathbb{Z}$ is nothing but the
manifold $\mathbb{T}_{A}^{3}$ considered in \cite{ghys-sergiescu80},
where
\begin{equation*}
A=\left(
    \begin{array}{cc}
      0 & 1 \\
      -1 & 0 \\
    \end{array}
  \right).
\end{equation*}
As, according to \cite{ghys-sergiescu80}, $A$ is neither conjugated with a matrix of the form
\begin{equation*}
A=\left(
    \begin{array}{cc}
      1 & n \\
      0 & 1 \\
    \end{array}
  \right)
\end{equation*}
nor
\begin{equation*}
A=\left(
    \begin{array}{cc}
      -1 & n \\
      0 & 1 \\
    \end{array}
  \right)
\end{equation*}
the first Betti number of $M_f$ is $1$. Therefore, $M_f$ cannot be
the global product of a compact surface $S$ with the unit circle
$S^1$. Indeed in such a case the first Betti number of $S$ would be
$0$ and hence $S$ would be the sphere $S^2$. Then from the homology
exact sequence
\begin{equation*}
0 \longrightarrow \mathbb{Z}^2\longrightarrow \Pi_{1}(\mathbb{T}_{A}^{3}) \longrightarrow \mathbb{Z} \longrightarrow 0
\end{equation*}
one would get an injective homomorphism from $\mathbb{Z}^2$ into $\Pi_{1}(\mathbb{T}_{A}^{3})=\mathbb{Z}$, which yields a contradiction. This is the argument used
in \cite{chinea-deleon-marrero93} for proving that $M_f$ is not the global product of a K\"{a}hler manifold with $S^1$.

In \cite{marrero-padron98}, Marrero and Padr{\'o}n generalized this
example to any odd dimension. We briefly illustrate their
construction. Let us consider the $2n$-dimensional torus
$\mathbb{T}^{2n}$ with the K\"{a}hler structures $(J,G)$ and $(J',G')$
defined as follows
\begin{gather*}
J X_{i}=-Y_{i}, \ \ J Y_{i}=X_{i}, \ \ \ G=\sum_{j=1}^{2n}\left(\alpha_{j}\otimes\alpha_{j} + \beta_{j}\otimes\beta_{j}\right)\\
J' X'_{i}=-Y'_{i}, \ \ J' Y'_{i}=X'_{i}, \ \ \ G'=\sum_{j=1}^{2n}\left(\alpha'_{j}\otimes\alpha'_{j} + \beta'_{j}\otimes\beta'_{j}\right)
\end{gather*}
for each $i\in\left\{1,\ldots,n\right\}$, where  $\{X_{1},\ldots,X_{n},Y_{1},\ldots,Y_{n}\}$ is the canonical global basis of vector fields on $\mathbb{T}^{2n}$, $\{\alpha_{1},\ldots,\alpha_{n},\beta_{1},\ldots,\beta_{n}\}$ the corresponding dual basis of $1$-forms, $\alpha'_i:=\alpha_i+\cos \frac{\pi}{3} \beta_i$, $\beta'_i:=-\sin\frac{\pi}{3}\beta_i$ and $\{X'_{1},\ldots,X'_{n},Y'_{1},\ldots,Y'_{n}\}$ the dual basis of vector fields of the basis of $1$-forms $\{\alpha'_{1},\ldots,\alpha'_{n},\beta'_{1},\ldots,\beta'_{n}\}$.

Then, as described before, we can consider the action of
$\mathbb{Z}$ on $\mathbb{T}^{2n}\times\mathbb{R}$ via an Hermitian
isometry $f$ and the corresponding quotient space
$M_{f}:=(\mathbb{T}^{2n}\times\mathbb{R})/\mathbb{Z}$. Let us
consider the case when $f$ is induced by an isometry $\tilde{f}$ of
$\mathbb{R}^{2n}$.  Marrero and Padr{\'o}n considered the following
four cases
\begin{gather*}
\tilde{f}_{1}(x_{1},\ldots,x_{n},y_{1},\ldots,y_{n})=(y_{1},\ldots,y_{n},-x_{1},\ldots,-x_{n})\\
\tilde{f}_{2}(x_{1},\ldots,x_{n},y_{1},\ldots,y_{n})=(-x_{1},\ldots,-x_{n},-y_{1},\ldots,-y_{n})\\
\tilde{f}_{3}(x_{1},\ldots,x_{n},y_{1},\ldots,y_{n})=(-y_{1},\ldots,-y_{n},x_{1}+y_{1},\ldots,x_{n}+y_{n})\\
\tilde{f}_{4}(x_{1},\ldots,x_{n},y_{1},\ldots,y_{n})=(-x_{1}-y_{1},\ldots,-x_{n}-y_{n},x_{1},\ldots,x_{n}),
\end{gather*}
so obtaining the corresponding compact coK\"{a}hler manifolds $M_{f_{i}}$, $i\in\left\{1,\ldots,4\right\}$. Notice that, for $n=1$, $M_{f_1}$ coincides
with the $3$-dimensional example of Chinea, de Le\'{o}n, Marrero illustrated before.   Then the authors proved the following remarkable result:

\begin{theorem}[\cite{marrero-padron98}]
The manifolds $M_{f_1}$, $M_{f_2}$, $M_{f_3}$, $M_{f_4}$ are $(2n+1)$-dimensional compact flat coK\"{a}hler solvmanifolds which are not topologically
equivalent to the compact coK\"{a}hler manifolds $\mathbb{T}^{2m+1}\times \mathbb{C}P^{r}$, with $m,r>0$, $m+r=n$.
\end{theorem}

Concerning the above examples, Fino and Vezzoni (\cite{fino11}) showed that these solvmanifolds are finite quotients of a torus. In fact the authors prove
a  more general result, namely that a solvmanifold has a coK\"{a}hler structure
if and only if it is a finite quotient of  a torus which has a structure of a
torus bundle over a complex torus.
Recently, in \cite{bazzoni-oprea} Bazzoni and Oprea found non-trivial examples
of non-product compact coK\"ahler manifolds in every odd dimension.

\subsection{CoK\"{a}hler manifolds and K\"{a}hler geometry}
CoK\"{a}hler manifolds are the most natural analogues in odd
dimension of K\"{a}hler manifolds. This can be seen for instance by
looking at the $1$-dimensional foliation $\mathcal V$ defined by the
Reeb vector field, usually called \emph{Reeb foliation}. In fact
notice that the closedness of the fundamental $2$-form $\Phi$
together with \eqref{sasakiform1} implies that $\Phi$ is a basic
$2$-form. Moreover, by  Theorem \ref{killing} the vanishing of
$N^{(3)}_{\phi}$  also  implies that the tensor field  $\phi$ and
the Riemannian metric $g$ are  projectable tensors. Roughly
speaking, this means that $\Phi$, $\phi$ and $g$ are constant along
the leaves of the Reeb foliation, so that they locally project to a
$2$-form $\omega'$, a tensor field $J'$ and a Riemannian metric $G'$
on the leaf space. Then $(\omega',J',G')$ form a K\"{a}hler
structure. Thus any coK\"{a}hler manifold is foliated by a
transversely K\"{a}hler foliation. If in particular $\mathcal V$ is
simple, then there exists a Riemannian submersion $\pi:(M,g)
\longrightarrow (N',G')$ such that $\ker(\pi_\ast)=T{\mathcal V}$
and $\pi_{\ast}\circ\phi=J'\circ\pi_{\ast}$.

Therefore, at least locally, any coK\"{a}hler manifold projects onto a K\"{a}hler one. On the other hand, given a coK\"{a}hler manifold $M$ there exists a
K\"{a}hler manifold $N$ projecting onto $M$. It is enough to consider the product of $M$ with the real line (or the unit circle), endowed with the complex
structure $J$ defined by \eqref{complex} and the product metric $G=g + \d t^2$. An easy computation shows that $G$ is an Hermitian metric. Moreover, one
has
\begin{align*}
\omega\left(\left(X,f\frac{\d}{\d t}\right),\left(Y,h\frac{\d}{\d t}\right)\right)&=G\left(\left(X,f\frac{\d}{\d t}\right),J\left(Y,h\frac{\d}{\d t}\right)\right)\\
&=\Phi(X,Y)-h \eta(X) + f \eta(Y),
\end{align*}
from which it follows that
\begin{align*}
\d\omega\left(\left(X,f\frac{\d}{\d t}\right),\left(Y,h\frac{\d}{\d t}\right),\left(Z,k\frac{\d}{\d t}\right)\right)&=\d\Phi(X,Y,Z)-\frac{2}{3}\bigl(k \d\eta(X,Y) -
 h \d\eta(X,Z)+  f \d\eta(Y,Z)\bigr),
\end{align*}
for any $X,Y,Z\in\Gamma(TM)$ and $f,h,k\in C^{\infty}(M)$. Therefore, the almost Hermitian structure induced on $M\times\mathbb{R}$ is (almost) K\"{a}hler
if and only if $M$ is (almost) coK\"{a}hler.

This construction on the product of an almost contact manifold with the real line can be extended to a more general setting. Let
$(M_{1},\phi_{1},\xi_{1},\eta_{1})$ and $(M_{2},\phi_{2},\xi_{2},\eta_{2})$ be almost contact manifolds. We define on the product manifold $M_{1}\times
M_{2}$ a tensor field $J$ by setting
\begin{equation*}
J\left(X_{1},X_{2}\right)=\left(\phi_{1}X_{1}-\eta_{2}(X_{2})\xi_{1}, \phi_{2}X_{2}+\eta_{1}(X_{1})\xi_{2}\right).
\end{equation*}
One checks that $J^2 = -I$, so that an almost complex structure is defined on $M_{1}\times M_{2}$. This construction is due to Morimoto
(\cite{morimoto63}), who proved that $(M_{1}\times M_{2},J)$ is a complex manifold if and only if $(M_{1},\phi_{1},\xi_{1},\eta_{1})$ and
$(M_{2},\phi_{2},\xi_{2},\eta_{2})$ are normal almost contact manifolds. It is
worth mentioning that an interesting corollary of this  is the well-known
result, previously proved by Calabi and Eckmann, that the product
$S^{2p+1}\times S^{2q+1}$ of any two odd-dimensional spheres is a complex manifold.
Notice that, since the second Betti numbers of these compact manifolds are zero, they cannot carry any K\"{a}hler structure.

Now let us consider two almost contact metric manifolds $(M_{1},\phi_{1},\xi_{1},\eta_{1},g_{1})$ and $(M_{2},\phi_{2},\xi_{2},\eta_{2},g_{2})$. Then
Goldberg (\cite{goldberg68}) proved that the product $M_{1}\times M_{2}$, endowed with the above almost complex structure $J$ and with the product metric
$G=g_1 + g_2$ is K\"{a}hler, provided that the factors are coK\"{a}hler. Later on, Capursi (\cite{capursi84}) proved that also the converse holds. In fact
he proved the result in a more general setting, dealing with almost K\"{a}hler and nearly K\"{a}hler manifolds. Summing up, we can state the following
theorem.

\begin{theorem}
The almost Hermitian manifold $(M_{1}\times M_{2},J,G)$ is (almost) K\"{a}hler if and only if both $(M_{1},\phi_{1},\xi_{1},\eta_{1},g_{1})$ and
$(M_{2},\phi_{2},\xi_{2},\eta_{2},g_{2})$ are (almost) coK\"{a}hler.
\end{theorem}

In \cite{watson83} Watson further generalized Morimoto's construction by considering, for each real number $a$ and $b\neq 0$ the manifold $M=M_{1}\times
M_{2}$ endowed with the tensor field $J_{a,b}$ and the Riemannian metric $G_{a,b}$ defined by
\begin{equation*}
J_{a,b}\left(X_{1},X_{2}\right)=\left(\phi_{1}X_{1}-\left(\frac{a}{b}\eta_{1}(X_{1})+\frac{a^2+b^2}{b}\eta_{2}(X_{2})\right)\xi_{1},
\phi_{2}X_{2}+\left(\frac{1}{b}\eta_{1}(X_{1})+\frac{a}{b}\eta_{2}(X_{2})\right)\xi_{2}\right)
\end{equation*}
and
\begin{align*}
G_{a,b}\left((X_{1},X_{2}),(Y_{1},Y_{2})\right)&=g_{1}\left(X_{1},Y_{1}\right)+a\eta_{1}(X_{1})\eta_{2}(Y_{2}) + a\eta_{1}(Y_{1})\eta_{2}(X_{2}) \\
&\quad+ (a^2 + b^2 +1)\eta_{2}(X_{2})\eta_{2}(Y_{2}) + g_{2}\left(X_{2},Y_{2}\right).
\end{align*}
In the most general case, that is when
$(M_{1},\phi_{1},\xi_{1},\eta_{1},g_{1})$ and
$(M_{2},\phi_{2},\xi_{2},\eta_{2},g_{2})$ are almost contact metric
manifolds, one has that $(M,J_{a,b},G_{a,b})$ is an almost Hermitian
manifold. The fundamental $2$-form
$\omega_{a,b}:=G(\cdot,J_{a,b}\cdot)$ is then given by
\begin{equation*}
\omega_{a,b}=\Phi_1 + \Phi_2 + 2 b \eta_{1}\wedge \eta_{2},
\end{equation*}
from which
\begin{equation*}
\d\omega_{a,b}=\d\Phi_1 + \d\Phi_2 + 2 b\left(\d\eta_{1}\wedge \eta_{2}-\eta_{1}\wedge \d\eta_{2}\right).
\end{equation*}
Again, \ one \ proves \ that \ $(M,J_{a,b},G_{a,b})$ \ is \ a \
K\"{a}hler \ manifold \ if \ and only \ each \
$(M_{i},\phi_{i},\xi_{i},\eta_{i},g_{i})$ is coK\"{a}hler,
$i\in\left\{1,2\right\}$. Then Watson used the aforementioned
Olszak's  examples of strictly almost coK\"{a}hler non-compact
solvable Lie groups  for constructing new examples of almost
K\"{a}hler non-K\"{a}hler manifolds. It is enough to take the
product of any of such Lie groups and to apply the previous results
on the almost K\"{a}hler structure on the products. Namely, he
proved the following result.

\begin{theorem}
Given two odd integers, $2p + 1$ and $2q + 1$, $p, q > 0$, there exist noncompact, connected, solvable Lie groups $G_{1}^{2p+1}$ and $G_{2}^{2q+1}$, such
that $G=G_{1}^{2p+1} \times G_{2}^{2q+1}$ carries a two parameter family of almost K\"{a}hler structures which are not K\"{a}hler.
\end{theorem}

Further, one can consider Thurston's torus bundle $W^{4}$ over
$\mathbb{T}^{2}$ with strictly almost K\"{a}hler structure $(J, g)$. Then one
puts on  $W^{4} \times S^{1}$ the standard almost coK\"{a}hler
structure, so that $W^{4} \times S^{1}$  cannot be coK\"{a}hler
because $W^4$ is not K\"{a}hler. Thus, the non-K\"{a}hler almost
K\"{a}hler structure described before is defined  on $M = (W^{4} \times S^{1})
\times (W^{4} \times S^{1})$. Note that, as shown in
\cite{watson83}, we could not have obtained the non-K\"{a}hler
property of $M$ from its cohomology. Moreover, if $N$ is any compact
K\"{a}hler manifold, then $M = (W^{4}\times S^1) \times (N \times S^1)$
can be made strictly almost K\"{a}hler in the same way. However, in
this case, $b_{1}(M) = b_{1}(N) + 5$ is odd and the non-K\"{a}hler
property emerges from cohomological considerations.

\medskip

Turning to the general case, we mentioned that the $2n$-dimensional distribution ${\mathcal D}:=\ker(\eta)$ of any almost coK\"{a}hler manifold $M$ of
dimension $2n+1$ is integrable and so defines a foliation orthogonal to the Reeb foliation. Since $\phi X\in\Gamma({\mathcal D})$ for any
$X\in\Gamma({\mathcal D})$, $\phi$ induces an almost complex structure $J$ on each leaf of $\mathcal D$, which is compatible with the metric $G$ induced by
$g$. One can easily see that the fundamental $2$-form of this almost Hermitian structure is closed, so that we have that the leaves of $\mathcal D$ are
almost K\"{a}hler manifolds. Olszak  proposed the following definition (\cite{olszak87}). An \emph{almost coK\"{a}hler manifold with K\"{a}hlerian leaves}
is an almost coK\"{a}hler manifold such that the leaves of the canonical foliation $\mathcal D$ are K\"{a}hler manifolds. Clearly if  $M$ is coK\"{a}hler
or if $\dim(M)=3$ then $M$ is an almost coK\"{a}hler manifold with K\"{a}hlerian leaves. Thus the question is whether in dimension greater than $3$ any
almost coK\"{a}hler manifold with K\"{a}hlerian leaves is necessarily coK\"{a}hler. Olszak gave a negative answer, by finding some explicit counterexamples (see
Example 3 and 4 in \cite{olszak87}). Moreover he proved the following characterization.

\begin{theorem}[\cite{olszak87}]
Let $(M,\phi,\xi,\eta,g)$ be an almost coK\"{a}hler manifold. Then $M$ is almost coK\"{a}hler with K\"{a}hlerian leaves if and only if the structure
$(\phi,\xi,\eta,g)$ fulfils the condition
\begin{equation*}
(\nabla_{X}\phi)Y=g(X,hY)\xi-\eta(Y)hX
\end{equation*}
for any $X,Y\in\Gamma(TM)$.
\end{theorem}

\begin{corollary}\label{killing2}
An almost coK\"{a}hler manifold with K\"{a}hlerian leaves is coK\"{a}hler if and only if the Reeb vector field is Killing.
\end{corollary}

We remarked that in dimension $3$ any almost coK\"{a}hler manifold with $\xi$
Killing is necessarily coK\"{a}hler, but in higher dimensions this is not true.
Thus in view of Corollary \ref{killing2}, we can say that the condition of having K\"{a}hlerian leaves is what is missing for an almost coK\"{a}hler manifold with $\xi$ Killing to be coK\"{a}hler.

Furthermore, we mention that recently Dacko has developed a new approach towards almost coK\"{a}hler manifolds with K\"{a}hlerian leaves based on the techniques of CR-geometry (see \cite{dacko2012-2} for more details).

\subsection{CoK\"{a}hler structures on hypersurfaces of K\"{a}hler manifolds}
Other examples come from hypersurfaces of K\"{a}hler manifolds. The first author that studied the geometric structures induced on a hypersurface of an
almost Hermitian manifold was Tashiro in 1963 (\cite{tashiro1963-I}, \cite{tashiro1963-II}). A few years later Okumura (\cite{okumura65}) and Goldberg
\cite{goldberg68} found necessary and sufficient conditions for which a hypersurface of a K\"{a}hler manifold carries a coK\"{a}hler structure. Let
$(\tilde{M},\tilde{J},\tilde{g})$ be a $(2n+2)$-dimensional K\"{a}hler manifold and let $\iota:M \longrightarrow \tilde{M}$ be a $C^{\infty}$ orientable
hypersurface. Let $N$ be the unit normal of $M$ with orientation determined by that of $M$. Then for any vector field $X\in\Gamma(TM)$ let us decompose
$\tilde{J}\iota_{\ast}X$ into the components tangent and normal to $M$, so defining a $(1,1)$-tensor field $\phi$ and $1$-form $\eta$ on $M$ by
\begin{equation*}
\tilde{J}\iota_{\ast}X=\iota_{\ast}\phi X + \eta(X)N.
\end{equation*}
Next, one can prove that $\xi:=-\tilde{J}N$  is a vector field tangent to $M$ and that $(\phi,\xi,\eta,g)$ is an almost contact metric structure, where
$g:=\iota^{\ast}\tilde{g}$ is the induced metric. Now we have the following theorem:

\begin{theorem}[\cite{okumura65}, \cite{goldberg68}]
With the above notation, a necessary and sufficient condition for a
smooth orientable hypersurface $M$ of a K\"{a}hler manifold
$\tilde{M}$ to be almost coK\"{a}hler is that its Weingarten operator
$A$ anticommutes with $\phi$,
\begin{equation*}
A\phi+\phi A =0.
\end{equation*}
A necessary and sufficient condition for  $M$  to be coK\"{a}hler
is that its second fundamental form $h$ is proportional to
$\eta\otimes\eta$, that is
\begin{equation*}
h=f \eta\otimes\eta,
\end{equation*}
where $f=h(\xi,\xi)$, or equivalently
\begin{equation*}
A=f \eta\otimes\xi.
\end{equation*}
\end{theorem}

In particular it follows that any totally geodesic $C^{\infty}$ orientable hypersurface of a K\"{a}hler manifold carries a coK\"{a}hler structure. Further,
Okumura studied the geometry of the hypersurfaces in a K\"{a}hler manifold of constant holomorphic sectional curvature. He proved that there is no
coK\"{a}hler hypersurface in a K\"{a}hler manifold of positive constant holomorphic sectional curvature. This result was improved by Olszak in \cite{olszak82},
where the following theorem was proved.

\begin{theorem}[\cite{olszak82}]
There are no almost coK\"{a}hler hypersurfaces in K\"{a}hler manifolds of non-zero constant holomorphic sectional curvature.
\end{theorem}

\subsection{Homogeneous coK\"{a}hler manifolds}
An almost contact metric manifold $(M,\phi,\xi,\eta,g)$ is said to be \emph{almost contact homogeneous} if there exists a connected Lie group $G$ of
isometries acting transitively on $M$ and leaving $\phi$ (and hence also $\eta$ and $\xi$, cf. \cite{chinea-gonzalez86}) invariant.

Studying almost contact homogeneous manifolds is a very difficult task in the full generality, so one restricts to some remarkable classes of almost
contact metric manifolds. In fact in the coK\"{a}hler setting there is a full classification in dimension $3$, due to a recent work of Perrone
(\cite{perrone12}). Namely,  a $3$-dimensional simply connected homogeneous almost coK\"{a}hler manifold is either a Riemannian product of the real line
with a K\"{a}hler surface of constant curvature or a Lie group $G$ equipped with a left invariant almost coK\"{a}hler structure. Moreover, in this last
case, we have the following classification, based on the invariant
\begin{equation*}
p:=\left\|{\mathcal L}_{\xi}h\right\| - 2\left\|h\right\|^2.
\end{equation*}
\emph{CoK\"{a}hler case}
\begin{enumerate}
\item[(1)] If $G$ is unimodular, then the metric is flat and $G$ is either the universal covering $\tilde{E}(2)$ of the group of rigid motions of Euclidean $2$-space, or the
abelian Lie group $\mathbb{R}^3$.
\item[(2)] If $G$ is non-unimodular, its Lie algebra $\mathfrak{g}:=\langle e_{1}, e_{2}, e_3 \rangle$ is given by
\begin{equation*}
[e_{1},e_{2}]=\alpha e_{2}, \ \ \ [e_{1},e_{3}]=[e_{2},e_{3}]=0, \ \ \ (\alpha\neq 0)
\end{equation*}
where $e_{1},e_{2}=\phi e_1 \in{\mathcal D}$ and $e_3 =\xi$.
\end{enumerate}
\emph{Non-coK\"{a}hler case}
\begin{enumerate}
\item[(3)] If $G$ is unimodular, then $G$ is one among the following:
\begin{enumerate}
\item[-] $\tilde{E}(2)$, if $p>0$
\item[-] the group $E(1,1)$ of rigid motions of the Minkowski $2$-space, if $p<0$
\item[-] the Heisenberg group $H^3$, if $p=0$.
\end{enumerate}
\item[(4)] If $G$ is non-unimodular, then $p = 0$ and its Lie algebra $\mathfrak{g}:=\langle e_{1}, e_{2}, e_3 \rangle$ is given by
\begin{equation*}
[e_{1},e_{2}]=\alpha e_{2}, \ \ \ [e_{1},e_{3}]=\gamma e_{2} \ \ \ [e_{2},e_{3}]=0, \ \ \ (\alpha,\gamma\neq 0)
\end{equation*}
where $e_{1},e_{2}=\phi e_1 \in{\mathcal D}$ and $e_3 =\xi$.
\end{enumerate}
Furthermore, Perrone proves that in all the aforementioned cases
except (4) the Reeb vector field $\xi$ is harmonic. Recall that a
unit vector field $U$ on a $m$-dimensional Riemannian manifold
$(M,g)$ can be regarded as a map $U:(M,g) \longrightarrow
(T_{1}M,g_S)$, where $g_S$ denotes the Sasaki metric on the unit
tangent sphere bundle $T_{1}M$. When $M$ is compact we can define
the \emph{energy} of $U$  as the energy of the corresponding map:
\begin{equation*}
E(U)=\frac{1}{2}\int_{M}\left\|dU\right\|^2\nu_{g}=\frac{m}{2}\textrm{vol}(M,g) + \frac{1}{2}\int_{M}\left\|\nabla U\right\|^2\nu_{g}.
\end{equation*}
A unit vector field $U$ is then called \emph{harmonic} if it is  critical for the energy functional $E$ defined on the set of all unit vector fields on $M$.

\medskip

A family of examples of homogeneous almost coK\"{a}hler structures in dimension greater than $3$ is proposed in \cite{chinea-gonzalez86}. Let us consider the
generalized  Heisenberg group $H(1,r)$, $r>1$, that is the Lie group of real matrices of the form
\begin{equation*}
\left(
  \begin{array}{cccccc}
    1 & 0 & \cdots & 0 & a_{1}^{r+1} & a_{1}^{r+2} \\
    0 & 1 & \cdots & 0 & a_{2}^{r+1} & a_{2}^{r+2} \\
    \vdots & \vdots & \ddots & \vdots & \vdots & \vdots \\
    0 & 0 & \cdots & 1 & a_{r}^{r+1} & a_{r}^{r+2} \\
    0 & 0 & \cdots & 0 & 1 & a_{r+1}^{r+2} \\
    0 & 0 & \cdots & 0 & 0 & 1 \\
  \end{array}
\right),
\end{equation*}
where $a_{i}^{j}\in\mathbb{R}$. This is a connected and simply connected Lie group of dimension $2r+1$.  Let $\left\{x_{i}, y_{i}, z\right\}$,
$i\in\left\{1,\ldots,r\right\}$, be the coordinate functions on $H(1,r)$ defined by $x_{i}(A)=a_{i}^{r+1}$, $y_{i}(A)=a_{i}^{r+2}$, $z(A)=a^{r+2}_{r+1}$,
for any $A\in H(1,r)$. One proves that the $1$-forms $\alpha_{i}:=\d x_i$, $\beta_{i}:=\d y_{i}-x_{i}\d z$ and $\eta:=\d z$ are linearly independent and invariant
under the action of $H(1,r)$. Moreover, by setting
\begin{equation*}
\Phi:=-\sum_{i=1}^{r}\alpha_{i}\wedge \beta_{i}
\end{equation*}
it turns out that $(\eta,\Phi)$ is a cosymplectic structure. A compatible almost coK\"{a}hler structure is defined in the following way. Let
$X_{i}=\frac{\partial}{\partial x_i}$, $Y_{i}:=\frac{\partial}{\partial y_i}$, $\xi:=\frac{\partial}{\partial
z}+\sum_{i=1}^{r}x_{i}\frac{\partial}{\partial z}$ be the dual vector fields of the left invariant forms $\alpha_i$, $\beta_i$, $\eta$, respectively. Then
one can easily prove that they form an orthonormal frame field with respect to the left invariant metric on $H(1,r)$ defined by
\begin{equation*}
g:=\sum_{i=1}^{r}\left(\alpha_{i}^{2}+\beta_{i}^{2}\right) + \eta\otimes\eta.
\end{equation*}
Further we define a tensor field $\phi$ by imposing that $g(\cdot,\phi\cdot)=\Phi$. One can check that $(\phi,\xi,\eta,g)$ is an almost coK\"{a}hler
structure. Since $\nabla_{X_i}\xi\neq 0$, the structure is not coK\"{a}hler (cf. \cite{chinea-gonzalez86}). Chinea and Gonzalez give a detailed study of
homogeneous structures on $H(1,r)$. In particular, they prove that $H(1,r)$ is a homogeneous almost coK\"{a}hler manifold of type ${\mathcal
T}_{2}\oplus{\mathcal T}_{3}$, according to the classification  of homogeneous structures of Tricerri and Vanhecke (\cite{tricerri83}).

Further, we also notice that one can consider the manifold
$M(1,r)=\Gamma(1,r) \backslash H(1,r)$, where $\Gamma(1, r)$ is the
discrete subgroup of $H(1,r)$ of matrices with integer entries.
$M(1, r)$ is a compact nilmanifold of dimension $2r + 1$. However, we
point out that $M(1,r)$ cannot admit any coK\"{a}hler
structure due to topological obstructions (\cite{cordero85}).
Therefore, as the almost coK\"{a}hler structure on $H(1,r)$ given above descends
to $M(1,r)$, we see that
 $M(1,r)$ is an example of a compact almost coK\"{a}hler manifold which cannot
 have any coK\"{a}hler structure. Another example of a compact almost
 coK\"{a}hler manifold that cannot admit any coK\"{a}hler structure is $M(1,1)\times \mathbb{T}^{2r}$ (see \cite{cordero85} for more details).

We notice that some almost contact metric structures on $M(1,r)$ are described in  \cite{chinea-deleon-marrero-93tris}.

Another class of examples, due to de Le\'{o}n (\cite{deleon89}),  is given by the compact solvmanifolds $M(k)=D_{1}\backslash S_{1}$, where $S_{1}$ is a
$3$-dimensional solvable non-nilpotent Lie group and $D_{1}$ a discrete subgroup of $S_{1}$. Namely, let $S_{1}$ be the set of matrices of the form
\begin{equation*}
\left(
  \begin{array}{cccc}
    e^{kz} & 0 & 0 & x \\
    0 & e^{-kz} & 0 & y \\
    0 & 0 & 1 & z \\
    0 & 0 & 0 & 1 \\
  \end{array}
\right)
\end{equation*}
where $x,y,z\in\mathbb{R}$ and $k$ is a real number such that $e^{k}+e^{-k}$ is
an integer different from $2$. One proves that the forms $\d x-k z
\d z$, $\d y + k y \d z$ and $\d z$ constitute a basis of right invariant $1$-forms on $S_{1}$. Next, let $D_1$ be the discrete subgroup of $S_{1}$ such that the
quotient space $M(k)=D_{1}\backslash S_{1}$ is compact (\cite{auslander63}). Then there exists a global basis $\left\{\alpha,\beta,\gamma\right\}$ of
$1$-forms on $M(k)$ such that $\pi^{\ast}\alpha=\d x - k z \d z$,
$\pi^{\ast}\beta=\d y + k y \d z$, $\pi^{\ast}\gamma=\d z$, where $\pi: S_{1}\longrightarrow M(k)$
is the canonical projection. Thus one has $\d\alpha=-k \alpha\wedge \gamma$,
$\d\beta=k \beta\wedge\gamma$ and $\d\gamma=0$. Moreover, if
$\left\{X,Y,Z\right\}$ is the dual basis of vector fields to $\left\{\alpha,\beta,\gamma\right\}$, we have
\begin{equation*}
[X,Y]=0, \ \ \ [X,Z]=k X, \ \ \ [Y,Z]=- k Y.
\end{equation*}
Now, as observed by de Le\'{o}n, $M(k)$ cannot admit any
coK\"{a}hler structure, since otherwise $M(k)\times S^{1}$ would be
a K\"{a}hler manifold, but it is known that it cannot admit any
complex structure (cf. \cite{fernandez-gray90}). Nevertheless an
explicit almost coK\"{a}hler structure can be defined by putting
\begin{equation*}
\phi:=\alpha\otimes Y - \beta\otimes X, \ \ \ \xi:=Z, \ \ \ \eta:=\gamma, \ \ \ g:=\alpha^2 + \beta^2 + \gamma^2.
\end{equation*}
Notice that the Reeb vector field is not Killing, otherwise the manifold would
be coK\"{a}hler in view of Remark \ref{3D-Killing}.

A generalization of this example in higher dimensions is considered in \cite{chinea-deleon-marrero-93tris}.

\medskip

We close the section by illustrating a general method for defining left invariant coK\"{a}hler structures on Lie groups. First recall the notion of
\emph{coK\"{a}hler Lie algebra}. Namely, a coK\"{a}hler Lie algebra is a  real Lie algebra $\mathfrak{g}$ of dimension $2n+1$ endowed with a $1$-form
$\eta$, an endomorphism $\phi$ and an inner product $g$  satisfying the conditions
\begin{gather*}
\eta(\xi)=1, \ \ \ \d\eta=0, \ \ \ \d\Phi=0\\
\phi^2 = -I + \eta\otimes\xi, \ \ \ N^{(1)}_{\phi}=0, \ \ \ g(\phi\cdot,\phi\cdot)=g - \eta\otimes\eta,
\end{gather*}
where $\xi$ is the dual vector field of $\eta$, $\Phi$ is the $2$-form defined by $\Phi:=g(\cdot,\phi\cdot)$ and $N^{(1)}_{\phi}$ is the tensor field formally
defined as in \eqref{n1}.

This notion is closely related to that of a \emph{K\"{a}hler Lie algebra}, a
real Lie algebra of dimension $2n$ endowed with a $2$-form $\omega$ and
 an endomorphism $J$ satisfying
\begin{gather*}
\d\omega=0, \ \ \ J^2 = -I, \ \ \ [J,J]=0, \ \ \ g(J\cdot,J\cdot)=g.
\end{gather*}

Concerning coK\"{a}hler Lie algebras, Fino and Vezzoni proved the following characterization.

\begin{theorem}[\cite{fino11}]
CoK\"{a}hler Lie algebras of dimension $2n + 1$ are in one-to-one correspondence
with $2n$-dimensional K\"{a}hler Lie algebras equipped with a skew-adjoint derivation $D$
commuting with the complex structure.
\end{theorem}

It follows in particular that any coK\"{a}hler unimodular Lie group is necessarily flat and solvable.

Actually, the above result was proved at the Lie group level by Dacko in \cite{dacko99} for almost coK\"{a}hler manifolds. His results can be summarized as
follows.

\begin{theorem}[\cite{dacko99}]
Let $(\phi,\xi,\eta,g)$ be a left-invariant almost coK\"{a}hler structure on a simply connected and connected Lie group $G$. Then $G$ is a semi-direct
product of the form $\mathbb{R}\times_{f} G_{1}$, where $G_{1}$ is a Lie group and $f$ is a $1$-parameter group of automorphisms of $G_{1}$, and
$(\phi,\xi,\eta,g)$ is an almost coK\"{a}hler extension of some almost K\"{a}hler structure on $G_1$.
\end{theorem}

\subsection{Curvature of  coK\"{a}hler manifolds}
Curvature properties of  coK\"{a}hler manifolds are among the most studied topics in coK\"{a}hler geometry.  We start by collecting some basic
properties of almost coK\"{a}hler manifolds. We adopt the following notation. We denote by $R$ the curvature tensor defined by $R(X,Y,Z,W):=-g(R(X,Y)Z,W)$, by
$\textrm{Ric}$ the Ricci tensor, given
 by $\textrm{Ric}(X,Y):=\sum\limits_{i=0}^{2n}R(E_{i},X,Y,E_{i})$,  where  $\left\{E_{0}:=\xi,E_{1},\ldots,E_{2n}\right\}$  is  a  local  orthonormal  basis.   Furthermore we denote
by $\textrm{Ric}^{\ast}$ the Ricci $\ast$-tensor,  defined by  $\textrm{Ric}^{\ast}(X,Y):=\sum\limits_{i=0}^{2n}R(E_{i},X,\phi Y,\phi E_{i})$, and,
similarly,
 $s^{\ast}$ will be the scalar  $\ast-$curvature, given by the trace of $\textrm{Ric}^{\ast}$. Moreover we shall denote by $K$ the sectional curvature,
 defined by $K(\pi)=K(X,Y):=R(X,Y,X,Y)$, where $X$ and $Y$ are orthonormal vectors spanning the $2$-plane $\pi\subset T_{x}M$.

 The first systematic study of the curvature of an almost coK\"{a}hler manifold is due to Olszak
(\cite{olszak81}). He found useful formulas relating the curvature tensor of an
almost coK\"{a}hler manifold with other tensor fields, such as the covariant
derivative of $\phi$, etc.
\begin{theorem}[\cite{olszak81}]
In any almost coK\"{a}hler manifold  the following identities  hold
\begin{gather}
R(\xi,X)\xi - \phi R(\xi,\phi X)\xi + 2h^{2}X =0 \label{curvature1},\\
s - s^{\ast} - \emph{Ric}(\xi,\xi) + \frac{1}{2}\left\|\nabla\phi\right\|^2=0
\label{curvature2},\\
\emph{Ric}(\xi,\xi)+\left\|\nabla\xi\right\|^2 =0, \label{curvature3}
\end{gather}
where $h = \frac12 \mathcal{L}_\xi \phi$.
Furthermore,  any compact almost contact metric manifold $(M,\phi,\xi,\eta,g)$  satisfying \eqref{curvature2} and \eqref{curvature3} is necessarily almost
coK\"{a}hler.
\end{theorem}

A variation of the last sufficient condition is the following.

\begin{theorem}[\cite{olszak81}]
Let $(M,\phi,\xi,\eta,g)$ be an almost contact metric manifold such that the integral curves of the Reeb vector field are geodesics and the forms $\eta$
and $\Phi$ are coclosed. If $M$ fulfills the conditions \eqref{curvature2} and \eqref{curvature3}, then it is almost coK\"{a}hler.
\end{theorem}

Now, one can prove (\cite{olszak81}) that in any $3$-dimensional
almost contact metric manifold one has $\left\|\nabla\phi\right\|^2
= 2\left\|\nabla\xi\right\|^2$, $3 \left\|\d\Phi\right\|^2 =
2\left(\delta\eta\right)^2$ and
$\left\|\delta\Phi\right\|^2=2\left\|\d\eta\right\|^2$. Thus we can
state the following theorems.

\begin{theorem}[\cite{olszak81}]
Any $3$-dimensional almost contact metric manifold is coK\"{a}hler if and only if the Reeb vector field is parallel.
\end{theorem}

\begin{theorem}[\cite{olszak81}]
Let $(M,\phi,\xi,\eta,g)$ be a $3$-dimensional almost contact metric manifold such that any among the following conditions is satisfied:
\begin{enumerate}
\item[(i)] At least one among $\eta$ or $\Phi$ is a harmonic form.
\item[(ii)] $M$ is compact and \eqref{curvature3} holds.
\end{enumerate}
Then $M$ is almost coK\"{a}hler.
\end{theorem}

Blair proved   that a coK\"{a}hler manifold of constant curvature is
locally flat (\cite{blair67}). Later on, Goldberg and Yano  proved
in \cite{goldberg-yano69} that an almost coK\"{a}hler manifold of
constant curvature is coK\"{a}hler if and only if it is locally
flat.  Finally, Olszak (\cite{olszak81}, \cite{olszak87}) showed
that almost coK\"{a}hler manifolds of non-zero constant curvature do
not exist in any dimension. Summing up, we can state the following
important theorem.

\begin{theorem}\label{flat}
If an almost coK\"{a}hler  manifold $M$ is of constant curvature, then the metric is flat and the structure is
coK\"{a}hler.
\end{theorem}

Turning to coK\"{a}hler manifolds, we state the following
properties.

\begin{theorem}[\cite{goldberg-yano69}]\label{curvature4}
In any coK\"{a}hler manifold $(M,\phi,\xi,\eta,g)$ the following properties hold
\begin{enumerate}
\item[(i)] $R(\phi X, \phi Y)=R(X,Y)$,
\item[(ii)] $K(\phi X, \phi Y) = K(X,Y)$,
\item[(iii)] $g(R(X,\phi X)Y,\phi Y)=-K(X,Y)-K(X,\phi Y)$, \ where $X$ and $Y$ are orthonormal vector fields.
\end{enumerate}
\end{theorem}

From Theorem \ref{curvature4} it follows that
\begin{equation}\label{curvature6}
\textrm{Ric}(\phi X,\phi Y)=\textrm{Ric}(X,Y)
\end{equation}
for any $X,Y\in\Gamma(TM)$. In particular we have that $\textrm{Ric}(X, \xi) = \textrm{Ric}(\phi X, \phi\xi) = 0$ and $\textrm{Ric}(X, \phi Y ) =
\textrm{Ric}(\phi X,-Y + \eta(Y)\xi) = -\textrm{Ric}(\phi X, Y )$. This last property implies that the tensor $\rho(X, Y ) = \textrm{Ric}(X, \phi Y )$ is
in fact a $2$-form.

By putting $Y=\xi$ in (i) and (ii) of Theorem \ref{curvature4}, we get, respectively,
\begin{gather*}
R(X,\xi)=0, \ \ \  K(X,\xi)=0
\end{gather*}
for any $X\in\Gamma(TM)$. Moreover,  as an immediate consequence of Theorem \ref{cokahlercondition1}, one has that
\begin{equation}\label{cokahlercondition2}
R(X,Y)\phi Z = \phi R(X,Y)Z
\end{equation}
for any vector fields $X,Y,Z$. An important result due to  Goldberg and Yano ensures that the converse also holds.

\begin{theorem}[\cite{goldberg-yano69}]
If the curvature transformation of the metric $g$ of the almost coK\"{a}hler manifold $(M,\phi,\xi,\eta,g)$ commutes with $\phi$, then $M$ is normal, that
is, it is a coK\"{a}hler manifold.
\end{theorem}

From \eqref{cokahlercondition2} it easily follows that
\begin{equation}\label{cokahlercondition3}
R(X,Y)\xi=0.
\end{equation}
Actually \eqref{cokahlercondition3} holds true for any almost coK\"{a}hler manifold with $\xi$ Killing, since $\nabla\xi=0$. We now prove that this is in
fact a characterization of this class of almost coK\"{a}hler manifolds.

\begin{theorem}
Let $(M,\phi,\xi,\eta,g)$ be an almost coK\"{a}hler manifold. \ Then the Reeb vector field $\xi$ is Killing if and only if $R(X,Y)\xi=0$ for any
$X,Y\in\Gamma(TM)$.
\end{theorem}
\begin{proof}
Clearly if $\xi$ is Killing then $R(X,Y)\xi=0$, as pointed out before. Conversely, let us assume that $R(X,Y)\xi=0$ for any $X,Y\in\Gamma(TM)$. Then in
particular one has $R(\xi,X)\xi =0$, where $X$ in any vector field orthogonal to $\xi$.  Then by applying \eqref{curvature1} we get $h^2 X =0$. This in turn
implies, since $g$ is positive definite, that $h X =0$. As $h\xi=0$, we conclude that the tensor field $h$ vanishes identically, so that, by Theorem
\ref{killing}, $\xi$ is a Killing vector field.
\end{proof}

\begin{corollary}\label{riccikilling}
In any almost coK\"{a}hler manifold with $\xi$ Killing one has $\emph{Ric}^{\sharp}\xi = 0$, where  $\emph{Ric}^{\sharp}$ denotes the Ricci operator.
\end{corollary}

Now, let $X$, $Y$, $Z$ any vector fields on a coK\"{a}hler manifold $(M,\phi,\xi,\eta,g)$ and decompose $Z$ as $Z=Z_{\mathcal D}+\eta(Z)\xi$, according to
the orthogonal decomposition $TM=\mathcal D \oplus \mathbb{R}\xi$. Then, since $\nabla{\mathcal D}\subset{\mathcal D}$,
$\eta(R(X,Y)Z)=\eta(R(X,Y)Z_{\mathcal D}) + \eta(Z)\eta(R(X,Y)\xi)=0$. Using this, \eqref{cokahlercondition2} and (i) of Theorem \ref{curvature4}, we
obtain
\begin{equation}\label{curvature5}
R(\phi X, \phi Y, \phi Z, \phi W) = R(X,Y,Z,W)
\end{equation}
for all $X,Y,Z,W\in\Gamma(TM)$. Using \eqref{curvature5} one  proves the following theorem (cf. \cite{endo1994-2}).

\begin{theorem}
On any coK\"{a}hler manifold the Ricci operator $\emph{Ric}^{\sharp}$ commutes with $\phi$.
\end{theorem}
\begin{proof}
Let us consider a $\phi$-basis $\{E_{1},\ldots,E_{2n},\xi\}$, where for each $i\in\left\{1,\ldots,n\right\}$ $E_{n+i}=\phi E_{i}$. Then, for any
$X,Y\in\Gamma(TM)$,
\begin{align*}
g(\textrm{Ric}^{\sharp} \phi X, \phi Y)&=\sum_{i=1}^{2n}R(E_{i},\phi X, \phi Y, E_{i}) + R(\xi,\phi X, \phi Y, \xi) \\
&=\sum_{i=1}^{2n}R(\phi E_{i},\phi X, \phi Y, \phi E_{i}) - g(R(\xi,\phi X)\phi Y,\xi)\\
&=\sum_{i=1}^{2n}R( E_{i},X, Y, E_{i})+R(\xi, X, Y, \xi) \\
&=g(\textrm{Ric}^{\sharp}X,Y).
\end{align*}
By Corollary \ref{riccikilling}, $\textrm{Ric}^{\sharp} \xi =0$, so that
$\eta(\textrm{Ric}^{\sharp}X)=g(\textrm{Ric}^{\sharp}X,\xi)=g(\textrm{Ric}^{\sharp}\xi,X)=0$. Hence we have
$-g(\phi \textrm{Ric}^{\sharp}\phi X, Y) =g(\textrm{Ric}^{\sharp}X,Y)$, from which it follows that
\begin{equation*}
\phi \textrm{Ric}^{\sharp}\phi X = -\textrm{Ric}^{\sharp}X.
\end{equation*}
Applying $\phi$ yields $\phi \textrm{Ric}^{\sharp}X=-\phi^2 \textrm{Ric}^{\sharp}\phi X = \textrm{Ric}^{\sharp}\phi X - \eta(\textrm{Ric}^{\sharp}\phi
X)\xi =\textrm{Ric}^{\sharp}\phi X$.
\end{proof}

Theorem \ref{flat} shows that  the requirement of constancy of the
sectional curvature is too strong for a coK\"{a}hler manifold. This
leads us to introduce the notion of $\phi$-sectional curvature (also
called $\phi$-holomorphic curvature). Given a manifold $M$ with
coK\"{a}hler structure $(\phi,\xi,\eta,g)$, we call a \emph{$\phi$-section}
any plane section of $T_{x}M$ spanned by $\left\{X,\phi X\right\}$,
where $X$ is a vector orthogonal to $\xi$. The sectional curvature
$H(X):=K(X,\phi X)$ corresponding to such a section is then called
\emph{$\phi$-sectional curvature}. Eum (\cite{eum72}, see also
\cite{ludden70}) proved the following important theorem

\begin{theorem}
If a coK\"{a}hler manifold has a constant $\phi$-sectional curvature at every point, then the curvature tensor $R$ of the manifold is given by
\begin{align} \label{phisectional1}
R(X,Y)Z& = \frac{c}{4}\bigl(g(X,Z)Y - g(Y,Z)X + g(\phi X,Z)\phi Y - g(\phi Y,Z)\phi X \\
&\quad + 2g(\phi X,Y)\phi Z - \eta(X)\eta(Z)Y + \eta(Y)\eta(Z)X   +
\left(\eta(X)g(Y,Z)-\eta(Y)g(X,Z)\right) \xi \bigr),\nonumber
\end{align}
where $c$ is constant.
\end{theorem}

CoK\"{a}hler manifolds with constant $\phi$-sectional curvature are usually referred as \emph{coK\"{a}hler space forms} $M(c)$. There are several papers
dealing with coK\"{a}hler space forms, in many different contexts. We mention the result due to de Le\'{o}n and Marrero that given a $(2n+1)$-dimensional
compact coK\"{a}hler manifold $M$ with positive constant $\phi$-sectional curvature $c$, there exists a diffeomorphism  $f:M\longrightarrow
P_{n}(\mathbb{C})(c)\times S^{1}$, of $M$ onto the product of $S^{1}$ with the complex projective space of positive constant holomorphic sectional
curvature $c$, which is  transversally holomorphic isometric. Recall that a diffeomorphism $f:M \longrightarrow M'$ between two coK\"{a}hler manifolds is
said to be transversally holomorphic isometric if $f^{\ast}\eta'=\eta$ and for all $x\in M$ the mapping $f|_{L_x}:L_{x}\longrightarrow L'_{f(x)}$ is a
Hermitian isometry between the K\"{a}hler manifolds $L_{x}$ and $L'_{f(x)}$, where $L_{x}$ is the leaf of $\mathcal D$ passing through $x$ and where
$L'_{f(x)}$ is the leaf of ${\mathcal D}'$ over $f(x)$.

\begin{remark}
Notice that coK\"{a}hler space forms sit inside the more general framework of \emph{generalized Sasakian-space-forms} (see \cite{alegre-blair-carriazo2004} for more details).
\end{remark}

\medskip

By using \eqref{phisectional1}, one can prove that the Ricci tensor of a coK\"{a}hler space form $M(c)$ has the following expression
\begin{equation*}
\textrm{Ric}(X,Y) = \frac{n+1}{2}c\left(g(X,Y)-\eta(X)\eta(Y)\right).
\end{equation*}
More in general, if the Ricci tensor of an almost coK\"{a}hler manifold has the form
\begin{equation*}
\textrm{Ric}=a g + b \eta\otimes\eta,
\end{equation*}
where $a$ and $b$ are smooth functions, we say that $M$ is \emph{$\eta$-Einstein}. In particular if the Reeb vector field is Killing, by Corollary \ref{riccikilling} we get $a+b=0$. Notice that if $\xi$ is is Killing and the dimension $2n+1\geq 5$ then $a$ and $b$ are
necessarily constants. Thus we conclude that any almost coK\"{a}hler manifold such that the Reeb vector field is Killing is Einstein if and only if it is Ricci-flat. Regarding to this, we mention the following result, which can be considered as an odd-dimensional counterpart of the famous Goldberg conjecture.

\begin{theorem}[\cite{cappelletti-pastore10}]\label{einstein1}
Every compact Einstein almost coK\"{a}hler manifold with Killing Reeb vector field is necessarily coK\"{a}hler.
\end{theorem}

\begin{corollary}
Any compact Ricci-flat almost coK\"{a}hler manifold is coK\"{a}hler.
\end{corollary}

It follows that if $(M,\phi,\xi,\eta,g)$ is a compact coK\"{a}hler-Einstein manifold, then any other almost coK\"{a}hler structure $(\phi',\xi',\eta',g)$,
compatible with the same metric $g$, is necessarily normal. Indeed since $(M,g)$ is coK\"{a}hler-Einstein, it is Ricci-flat. Then by \eqref{curvature3}
$\|\nabla\xi'\|=0$ and so $\xi'$ is Killing. The assertion then follows from Theorem \ref{einstein1} applied to $(\phi',\xi',\eta',g)$.

Concerning $\eta$-Einstein manifolds, de Le\'{o}n and Marrero proved
the following result.

\begin{theorem}[\cite{deleon-marrero97}]
Every compact $\eta$-Einstein almost coK\"{a}hler manifold $M$ such that the Einstein constant $a > 0$ and  $\xi$ is a Killing vector field is a
coK\"{a}hler manifold. Furthermore, $M$ is locally a product of a simply connected compact K\"{a}hler-Einstein manifold and $S^{1}$, and $\pi_{1}(M)$ is
isomorphic to $\mathbb{Z}$.
\end{theorem}

\medskip

We close the section by mentioning an important generalization of the condition $R(X,Y)\xi=0$ that characterizes almost coK\"{a}hler manifolds with Killing
Reeb vector field.

Namely, given an almost contact metric manifold $(M,\phi,\xi,\eta,g)$ and two real numbers $\kappa$ and $\mu$ one can define a distribution $N(\kappa,\mu)$ on $M$
by
\begin{align*}
N_{x}(\kappa,\mu):=\{Z\in T_{x}M | R(X,Y)Z=\kappa (g(Y,Z)X&-g(X,Z)Y)+ \mu (g(Y,Z)hX-g(X,Z)hY)\}.
\end{align*}
The distribution $N(\kappa,\mu)$ is called the \emph{$(\kappa,\mu)$-nullity
distribution}.  In \cite{blair95}, the case when the Reeb vector field of a
contact metric manifold belongs to the $(\kappa,\mu)$-nullity distribution was considered. A few years later, the almost coK\"{a}hler case was also considered (\cite{endo2002}).

Thus an \emph{almost coK\"{a}hler $(\kappa,\mu)$-manifold} is an almost coK\"{a}hler manifold satisfying the curvature condition
\begin{equation}\label{kmucondition}
R(X,Y)\xi=\kappa \left(\eta(Y)X-\eta(X)Y\right) + \mu \left(\eta(Y)hX-\eta(X)hY\right),
\end{equation}
for some $\kappa,\mu\in\mathbb{R}$. In particular, any coK\"{a}hler
manifold satisfies such a condition for $\kappa=0$ and any $\mu$.
However, if $\kappa\neq 0$ any almost coK\"{a}hler
$(\kappa,\mu)$-manifold cannot be coK\"{a}hler, because of the
relation (see \cite{endo2002})
\begin{equation}\label{kmu1}
h^{2}=\kappa\phi^{2}.
\end{equation}
Notice also that from \eqref{kmu1} it follows that $\kappa\leq 0$
and $\kappa=0$ if and only if $M$ is coK\"{a}hler. In the sequel we
shall consider the non-coK\"{a}hler case, that is $\kappa < 0$. In
this case one proves that the operator $h$ admits $3$ distinct
eigenvalues, $0$ and $\pm\lambda:=\pm\sqrt{-\kappa}$, and the
tangent bundle of the manifold splits up as the orthogonal sum of
the corresponding eigendistributions
 $TM={\mathcal D}_{h}(\lambda)\oplus{\mathcal D}_{h}(-\lambda)\oplus{\mathcal D}_{h}(0)$,
where ${\mathcal D}_{h}(0)={\mathbb R}\xi$. Moreover, ${\mathcal D}_{h}(\lambda)$ and ${\mathcal D}_{h}(-\lambda)$ are totally geodesic foliations and are
related by $\phi{\mathcal D}_{h}(\pm\lambda)={\mathcal D}_{h}(\mp\lambda)$. This
decomposition allows us to prove that the condition \eqref{kmucondition} determines the curvature completely (cf. \cite[Theorem 3.1]{endo2002}). In particular it follows that
if $X$ is a unit vector field in ${\mathcal D}_{h}(\lambda)$ one has
\begin{equation*}
K(X,\xi)=\kappa+\lambda\mu, \ \ \ K(\phi X,\xi)=\kappa-\lambda\mu, \ \ \ K(X,\phi X)=-\kappa .
\end{equation*}
Furthermore, an explicit formula for the Ricci operator is provided
\begin{equation}\label{ricci}
\textrm{Ric}^{\sharp} = \mu h + 2n \kappa \eta\otimes\xi.
\end{equation}

\begin{remark}
From \eqref{ricci} it follows that $\textrm{Ric}^{\sharp}\xi=2n\kappa\xi$. On
the other hand, as proved in \cite{pak-kim-2006} (see also page 214 of
\cite{dragomir-perrone}), Theorem \ref{harmonicity} implies that the Reeb vector field of an almost coK\"{a}hler manifold is a harmonic vector field if and
only if it is an eigenvector of the Ricci operator. In view of this remark and Corollary \ref{riccikilling}, it follows that in any almost coK\"{a}hler
manifold such that $\xi$ is Killing or belongs to the $(\kappa,\mu)$-nullity distribution, the Reeb vector field is harmonic.
\end{remark}

Thus, in view of \eqref{ricci}, almost coK\"{a}hler $(\kappa,\mu)$-manifolds are $\eta$-Einstein if and only if $\mu=0$, that is
\begin{equation}\label{kcondition}
R(X,Y)\xi=\kappa \left(\eta(Y)X-\eta(X)Y\right).
\end{equation}
This case was studied, among others, by Dacko (\cite{dacko2000}), who proved
that $M$ is necessarily an almost coK\"{a}hler manifold with K\"{a}hlerian leaves and, moreover, gave a full description of the local structure of this class. We briefly
recall his construction, referring to \cite{dacko2000} for more details. Let $\lambda$ be a real positive number and $\mathfrak{g}_{\lambda}$ be the solvable non-nilpotent Lie algebra with basis $\{\xi,X_{1},\ldots,X_{n},Y_{1},\ldots,Y_{n}\}$ and non-zero
Lie brackets
\begin{equation*}
[\xi,X_{i}]=-\lambda X_{i}, \ \ \ [\xi,Y_{i}]=\lambda Y_{i},
\end{equation*}
for each $i\in\left\{1,\ldots,n\right\}$. Let $G_\lambda$ be a Lie group whose Lie algebra is $\mathfrak{g}_\lambda$ and let $(\phi,\xi,\eta,g)$ be the left-invariant almost
coK\"{a}hler structure defined by putting at the identity
\begin{gather}\label{klocal1}
g(X_{i},X_{j})=g(Y_{i},Y_{j})=\delta_{ij}, \  \ g(X_{i},Y_{j})=0, \ \ g(\xi,X_{i})=g(\xi,Y_{i})=0, \nonumber\\
\phi \xi=0, \ \ \ \phi X_{i}=Y_{i}, \ \ \ \phi Y_{i}=-X_{i}, \  \ \ \ \eta=g(\cdot,\xi).\label{klocal1add}
\end{gather}
Then one can prove that the almost coK\"{a}hler Lie group $(G_{\lambda},\phi,\xi,\eta,g)$ satisfies \eqref{kcondition}, with $\kappa=-\lambda^2$. Dacko
proved that locally all almost coK\"{a}hler $(\kappa,0)$-manifolds are of this type:

\begin{theorem}[\cite{dacko2000}]
Let $(M,\phi,\xi,\eta,g)$ be an almost coK\"{a}hler manifold whose Reeb vector field $\xi$ satisfies \eqref{kcondition}, for some $\kappa<0$. Then $M$ is
locally isomorphic to the above Lie group $G_{\lambda}$ endowed with the almost coK\"{a}hler structure \eqref{klocal1add}, where $\lambda=\sqrt{-\kappa}$.
\end{theorem}

It remains an open question whether it is possible to provide a full classification as in the contact metric case (cf. \cite{boeckx2000} and \cite{cappelletti2009}).


Finally, a couple of interesting generalizations of almost coK\"{a}hler
$(\kappa,\mu)$-manifolds are worth mentioning. The first is due to Dacko and Olszak (\cite{dacko-olszak2005-1}), who
defined an \emph{almost coK\"{a}hler $(\kappa,\mu,\nu)$-manifold} as an almost coK\"{a}hler manifold satisfying the following curvature condition
\begin{align}
R(X,Y)\xi=\kappa (\eta(Y)X-\eta(X)Y) + \mu (\eta(Y)hX&-\eta(X)hY) + \nu (\eta(Y)\phi hX-\eta(X)\phi hY) \label{kmunucondition}
\end{align}
where $\kappa,\mu,\nu$ belong to the subring of the ring of smooth functions $f$ on $M$ such that $df \wedge \eta=0$. Of course, taking $\nu=0$ and $\kappa$ and $\mu$ constant,
one reobtains the aforementioned notion of almost coK\"{a}hler $(\kappa,\mu)$-manifolds. Among other results, the authors proved that the leaves of the canonical foliations
$\mathcal D$ are locally flat K\"{a}hlerian manifolds, and then they found local models for the case $\kappa=1$ and $\nu=0$ (\cite{dacko-olszak2005-2}).

Very recently,  Carriazo and Mart\'{i}n-Molina provided an explicit description of the curvature tensor of an almost coK\"{a}hler $(\kappa,\mu,\nu)$-space and proposed a further generalization (see \cite{carriazo-martinmolina} for more details).

\subsection{Topology of compact coK\"{a}hler manifolds}\label{topology}
An interesting topic
is the study of topological properties of coK\"ahler and almost coK\"ahler
manifolds. In particular, it is important to investigate the topological
obstructions to their existence.


To begin with, recall the result of Goldberg and Yano (Theorem
\ref{harmonicity}) that for any almost coK\"{a}hler manifold  the $1$-form $\eta$ and the
$2$-form $\Phi$ are harmonic. In view of Theorem~\ref{harmonicity}, it follows that the first and the second Betti
number of any compact almost coK\"{a}hler manifold are different from zero. Actually,
as an application of the formula at page 109 of \cite{goldberg-book},  one can prove that in any almost contact metric manifold such that $\xi$ is Killing
and $\eta$ is harmonic, the wedge product by $\eta$ preserves harmonicity. It follows that in any compact almost coK\"{a}hler manifold that satisfies any
of the equivalent conditions of Theorem \ref{killing}, the $3$-form $\eta\wedge\Phi$ is  harmonic and hence also the third Betti number is non-zero.

In fact if the structure is also normal, a
stronger result holds.

\begin{theorem}[\cite{blair-goldberg67}]\label{betti1}
All the Betti numbers of a compact coK\"{a}hler manifold are non-zero.
\end{theorem}

In order to prove Theorem \ref{betti1}, one considers  the operators $L$ and $\Lambda$, dual to each other, defined by $L\alpha := \alpha\wedge \Phi$ and $\Lambda \alpha :={\star} L {\star} \alpha$, for any
form $\alpha$, where $\star$ denotes the Hodge star isomorphism.
In the case when $M$ is
  coK\"{a}hler, Blair and Goldberg  (\cite{blair-goldberg67}) proved that $L$ commutes with the Laplace-Beltrami operator $\Delta:=d \delta + \delta d$. Using this, Theorem \ref{betti1} follows.

Several years later, Chinea, de Le\'{o}n and Marrero wrote a
fundamental paper on the topology of compact coK\"{a}hler manifolds
(\cite{chinea-deleon-marrero93}). An important role in their study
is played by the subspaces $\Omega^{p}_{H\xi}(M)$ and
$\Omega^{p}_{H\eta}(M)$ of the vector space $\Omega^{p}_{H}(M)$ of
harmonic $p$-forms, on a compact coK\"{a}hler manifold $M$ of
dimension $2n+1$, defined by
\begin{gather*}
\Omega^{p}_{H\xi}(M) = \left\{\alpha\in\Omega^{p}_{H}(M) | i_{\xi}\alpha=0 \right\}\\
\Omega^{p}_{H\eta}(M) = \left\{\alpha\in\Omega^{p}_{H}(M) | \eta\wedge\alpha=0 \right\}.
\end{gather*}
Then, the authors of \cite{chinea-deleon-marrero93} prove the
splitting
\begin{equation}\label{splitting}
\Omega^{p}_{H}(M)=\Omega^{p}_{H\xi}(M)\oplus\Omega^{p}_{H\eta}(M),
\end{equation}
and, moreover, that  the dimensions of the above subspaces are topological invariants of $M$.

Using this and the Poincar\'{e} duality, one obtains the following monotonicity property.

\begin{theorem}[\cite{chinea-deleon-marrero93}]\label{monotony}
On a compact coK\"{a}hler manifold one has
\begin{equation*}
b_{0}\leq b_{1}\leq \ldots \leq b_{n} = b_{n+1} \geq b_{n+2} \geq \ldots \geq b_{2n+1}.
\end{equation*}
\end{theorem}

Next, by studying complex harmonic forms,
Chinea, de Le\'on and Marrero
prove the coK\"{a}hler counterpart of the well-known property that the
odd-dimensional Betti numbers of
a compact K\"{a}hler manifold are even.

\begin{theorem}[\cite{chinea-deleon-marrero93}]\label{bettidiv2}
On a compact coK\"{a}hler manifold one has:
\begin{enumerate}
\item[(i)] For any odd integer $p$, the dimension of $\Omega^{p}_{H\xi}(M)$ is even.
\item[(ii)] For any $0\leq p \leq n$, the difference $b_{2p+1}-b_{2p}$ is even. In particular, the first Betti number of $M$ is odd.
\end{enumerate}
\end{theorem}

Moreover, in \cite{chinea-deleon-marrero93} an odd-dimensional version of the Hard Lefschetz Theorem is given. The authors first prove that the linear map
$L^{n-p}:\Omega^{p}_{H\xi}(M)\longrightarrow \Omega^{2n-p}_{H\xi}(M)$, $\alpha \longmapsto \alpha\wedge\Phi^{n-p}$, is an isomorphism. Using this and the decomposition \eqref{splitting}, they
proved the following coK\"{a}hler version of the Hard Lefschetz Theorem:

\begin{theorem}[\cite{chinea-deleon-marrero93}]
In any compact coK\"{a}hler manifold, for each $p\leq n$ the linear
map of the space $\Omega^{p}_{H}(M)$ into the space
$\Omega^{2n+1-p}_{H}(M)$ given by
\begin{equation*}
\alpha\in\Omega^{p}_{H}(M) \longmapsto i_{\xi}(\alpha\wedge\Phi)\wedge\Phi^{n-p} + \eta\wedge\alpha\wedge\Phi^{n-p} \in\Omega^{2n+1-p}_{H}(M)
\end{equation*}
is an isomorphism.
\end{theorem}

A $p$-form is said to be \emph{effective} if it belongs to the kernel of
$\Lambda = \star L \star $. This definition can be given for any almost contact metric manifold.
Next, denoting by $\bar\Omega^{p}_{H}(M)$ the space of harmonic effective $p$-forms and putting
  $\bar\Omega^{p}_{H\xi}(M):=\bar\Omega^{p}_{H}(M)\cap\Omega^{p}_{H\xi}(M)$,
$\bar\Omega^{p}_{H\eta}(M):=\bar\Omega^{p}_{H}(M)\cap\Omega^{p}_{H\eta}(M)$, one has the following theorem, relating such subspaces with the Betti numbers of $M$.

\begin{theorem}[\cite{chinea-deleon-marrero93}]
On a compact coK\"{a}hler manifold one has
\begin{equation*}
\dim \bar\Omega^{p}_{H\xi}(M) = b_{p} - b_{p-1}, \ \ \ \dim \bar\Omega^{p}_{H\eta}(M) = b_{p-1} - b_{p-2},
\end{equation*}
for any $p\leq n+1$. Thus the dimensions of $\bar\Omega^{p}_{H\xi}(M)$ and of $\bar\Omega^{p}_{H\eta}(M)$ are topological invariants.
\end{theorem}

Many of the results in \cite{chinea-deleon-marrero93} have been revisited in 2008 by Li (\cite{li08}) using different topological techniques. His work is
based on the notion of ``symplectic / K\"{a}hler mapping torus''. We recall the
precise definitions.  Let $f\in \textrm{Diff}(S)$ be a diffeomorphism on a
closed, connected manifold $S$. The \emph{mapping torus} $S_f$ is obtained
from $S \times [0, 1]$ by identifying the two ends via $f$, namely
\begin{equation*}
S_f = S \times [0,1] / \left\{(x,0)\sim (f(x),1) | x\in S\right\}.
\end{equation*}
It is known that $S_f$ is the total space of  a fiber bundle over $S^1$ with fiber $S$. Now, if $(S,\Omega)$ is a symplectic manifold and $f$ a
symplectomorphism, $S_f$ is called \emph{symplectic mapping torus}. If $(S,J,G)$ is a K\"{a}hler manifold and $f$ an Hermitian isometry, we say that $S_f$
is a \emph{K\"{a}hler mapping torus}.

The main result of Li was the following remarkable criteria for the
existence of cosymplectic and coK\"{a}hler structures in terms of
mapping tori.

\begin{theorem}[\cite{li08}]\label{mappingtorus}
Let $M$ be a closed manifold. Then we have:
\begin{enumerate}
\item[(i)] $M$ admits a cosymplectic structure if and only there exists a symplectic manifold $(S,\Omega)$ and a symplectomorphism $f$ of
$(S,\Omega)$ such that $M$ is diffeomorphic to $S_{f}$.
\item[(ii)]  $M$ admits a coK\"{a}hler structure if and only there exists a K\"{a}hler manifold $(S,J,G)$ and a Hermitian isometry $f$ of
$(S,J,G)$ such that $M$ is diffeomorphic to $S_{f}$.
\end{enumerate}
\end{theorem}

Theorem \ref{mappingtorus} can be regarded as an extension to the cosymplectic and coK\"{a}hler settings of the well-known result due to Tischler that any
compact manifold admitting a non-vanishing closed $1$-form fibres over a circle (\cite{tischler70}). Moreover, by using the second part of Theorem
\ref{mappingtorus}, Li was able to give a simpler proof of the aforementioned
monotonicity of the Betti numbers of a compact coK\"{a}hler manifold (Theorem
\ref{monotony}).

A key point in the proof of Theorem \ref{mappingtorus} (as well as in
the proof of Tischler's theorem) is that one can approximate the
$1$-form $\eta$ with a $1$-form $\eta_\theta$ representing an
integral cohomology class. This aspect is used in a very recent
paper of Bazzoni and Oprea (\cite{bazzoni-oprea}), where the
authors, using Li's characterization, give another type of structure
theorem for coK\"{a}hler manifolds. Namely, they prove the following
theorem.

\begin{theorem}[\cite{bazzoni-oprea}]\label{mappingtorus2}
A compact coK\"{a}hler manifold $(M,\phi,\xi,\eta,g)$ with integral $1$-form $\eta$ and mapping torus bundle $S \hookrightarrow M \longrightarrow S^1$
splits as
\begin{equation*}
M\cong S^{1}\times_{\mathbb{Z}_m} S,
\end{equation*}
where $S^{1}\times S \longrightarrow M$ is a finite cover with structure group $\mathbb{Z}_m$ acting diagonally  and by translation on the first factor.
Moreover, $M$ fibres over the circle $S^{1}/(\Z_m)$ with finite structure group.
\end{theorem}

Using Theorem \ref{mappingtorus2} Bazzoni and Oprea are able to recover, in a very simple way, some of the  Chinea, de Le\'{o}n, Marrero
results on the Betti numbers of  compact coK\"{a}hler manifolds. Moreover, they obtain a result regarding the fundamental group of a coK\"{a}hler manifold.
As we remarked, the  fundamental group of any coK\"{a}hler (actually, almost coK\"{a}hler) manifold is infinite. In fact we can say more:

\begin{theorem}[\cite{bazzoni-oprea}]\label{fundamentalgroup}
If $(M,\phi,\xi,\eta,g)$ is a compact coK\"{a}hler manifold with integral $1$-form $\eta$  and splitting $M\cong S^{1}\times_{\mathbb{Z}_m} S$, then
$\pi_{1}(M)$ has a subgroup of the form $H\times \mathbb{Z}$, where $H$ is the fundamental group of a compact K\"{a}hler manifold, such that the quotient
\begin{equation*}
\frac{\pi_{1}(M)}{H \times \mathbb{Z}}
\end{equation*}
is a finite cyclic group.
\end{theorem}

Theorem \ref{fundamentalgroup} is one of the few results concerning
the fundamental group of a coK\"{a}hler manifold. We mention another
one, due to de Le\'{o}n and Marrero. They studied in
\cite{deleon-marrero97} how some curvature conditions can affect the
topology of a compact coK\"{a}hler manifold. In particular they were
looking for an odd-dimensional analogue of the well-known fact that
a compact K\"{a}hler manifold with positive definite Ricci tensor is
simply connected. However, a compact simply connected manifold
cannot admit a coK\"{a}hler structure and, moreover, because of
\eqref{curvature3}, the Ricci tensor of a coK\"{a}hler manifold
cannot be positive definite. In particular, the Ricci tensor
$\textrm{Ric}$ of a coK\"{a}hler manifold of positive constant
$\phi$-sectional curvature is transversally positive definite, that
is for each point $x$ of the manifold $\textrm{Ric}_{x}$ is positive
definite on ${\mathcal{D}}_{x}$ (\cite{deleon-marrero97}). Now, we
have the following theorem.

\begin{theorem}[\cite{deleon-marrero97}]
    \label{pione}
Let $M$ be a compact coK\"{a}hler manifold with transversally positive definite Ricci tensor. Then the space $\Omega^{1}_{H\xi}(M)$ is trivial.
Moreover, $b_{1}=1$ and $\pi_{1}(M)$ is isomorphic to $\mathbb{Z}$. In particular, the first integral homology group of $M$ is also isomorphic to
$\mathbb{Z}$.
\end{theorem}
The $\pi_1$ part of Theorem~\ref{pione} was recently reproved
 in~\cite{bazzoni-oprea} using Li's approach.
Finally, concerning the almost coK\"{a}hler case, we mention the following result.

\begin{theorem}[\cite{deleon-marrero97}]
Let $M$ be a compact almost coK\"{a}hler manifold such that $b_{1} = 1$ and $\xi$ is Killing. Then the following statements hold:
\begin{enumerate}
\item[(i)]If $L$ is a leaf of the foliation $\mathcal{D}$, then the inclusion map induces a monomorphism
$i:\pi_{1}(L)\longrightarrow \pi_{1}(M)$ and the quotient group $\frac{\pi_{1}(M)}{\pi_{1}(L)}$ is isomorphic to $\mathbb{Z}$.
\item[(ii)] There exists a fibration $\pi:M\longrightarrow S^{1}$ such that the leaves of $\mathcal D$ are the fibres of $\pi$.
\item[(iii)] The leaves of $\mathcal D$ are compact.
\end{enumerate}
\end{theorem}

\subsection{Coeffective cohomology}\label{coeffective}
We say that a form $\alpha \in \Omega^p\left( M \right)$ in a
cosymplectic manifold $(M,\eta,\omega)$ is \emph{coeffective} if
$\alpha\wedge \omega =0$. Denote by $\mathcal{A}^p\left( M \right)$
the space of coeffective forms of degree $p$. Then since $\omega$ is
closed $\left( \mathcal{A}^p\left( M \right), \d\right)$ is a
subcomplex of the de Rham complex. The cohomology $H^p\left(
\mathcal{A}\left( M \right) \right)$ of the complex $\left(
\mathcal{A}^p\left( M \right),\d \right)$ is called
\emph{coeffective cohomology} of $M$. It was introduced
in~\cite{chinea-deleon-marrero95}.

The coeffective cohomology for cosymplectic manifolds was studied
in~\cite{chinea-deleon-marrero95, ibanez-fernandez97,fernandez98}.
We write $\widetilde{H}^p\left( M \right)$ for the subspace of $H^p\left( M
\right)$ of cohomology classes $\beta$ such that $\beta \wedge \left[ \omega
\right] =0$.

Chinea, de Le\'on and Marrero proved in
\cite{chinea-deleon-marrero93,chinea-deleon-marrero95}, that for a compact
coK\"ahler manifold $M^{2n+1}$
\begin{align*}
    H^p\left( \mathcal{A}\left( M \right) \right)&\cong 0, & \mbox{for } p&\le n-1\\
    H^p\left( \mathcal{A}\left( M \right) \right)&\cong
    \widetilde{H}^p\left( M \right), & \mbox{for }p&\ge n+2.
\end{align*}
In~\cite{fernandez98}, the authors construct an example of a compact
cosymplectic manifold for which the above result fails, thereby providing an
example of a cosymplectic manifold $(M,\eta, \omega)$ that does not admit any
coK\"ahler metric.

Let $M$ be a (possibly non-compact) cosymplectic manifold. We can
consider an almost coK\"ahler structure on $M$. Denote by $c_p$ the dimension of $H^p\left( \mathcal{A}\left( M \right)
\right)$. It was shown in~\cite{fernandez98} that
for any $p\ge n+2$
one has
\begin{equation*}
    b_p -b_{p+2} \le c_p     \le b_p + b_{p+2}.
\end{equation*}
In particular, every $c_p$ is finite. In the case $M$ is
a  compact coK\"ahler manifold, $c_p$ always attains
the lower bound $b_p-b_{p+2}$, $p\ge
n+2$. An example of a non-compact almost  coK\"ahler manifold $M$ such
that $c_p = b_p + b_{p+2}$, $p\ge n+2$ was also presented in~\cite{fernandez98}.

It is difficult to compute the coeffective cohomology for a general cosymplectic
manifold. In~\cite{ibanez-fernandez97} Fern{\'a}ndez et al.  provided a way to reduce
this problem to a purely algebraic one in the following special case.
Let $G$ be a connected nilpotent Lie group, $\mathfrak{g}$ its Lie algebra. Suppose that
$G$ is equipped with an invariant cosymplectic structure $\left( \eta,\omega\right)$ and
$\Gamma$ is a discrete subgroup of $G$ such that the space
of right cosets
$\left.\raisebox{-0.3ex}{$\Gamma$}\middle\backslash \raisebox{0.2ex}{$G$}\right. $
is compact.
We can consider the
Chevalley-Eilenberg complex $\left( \Omega^p\left( \mathfrak{g^*} \right),\d
\right)$ for $\mathfrak{g^*}$ as a subcomplex of left invariant forms in $\left(
\Omega^p\left( G \right), \d\right)$.
Denote by $\mathcal{A}^p\left( \mathfrak{g^*} \right)$
the intersection $\Omega^p\left( \mathfrak{g^*} \right)\cap \mathcal{A}^p\left(
G \right)$. Then $\left( \mathcal{A}^p\left( \mathfrak{g^*} \right),\d \right)$ is a
subcomplex of $\left( \Omega^p\left( \mathfrak{g^*} \right),\d \right)$.
Then
$$
H^p\left( \mathcal{A}\left(\mathfrak{g}^*  \right) \right)  \cong H^p\left(
\mathcal{A}\left( \left.\raisebox{-0.3ex}{$\Gamma$}\middle\backslash \raisebox{0.2ex}{$G$}\right. \right)
\right).
$$

\subsection{Rational homotopy type}
For a general and detailed treatment of rational homotopy theory the reader is referred to~\cite{felix08}.

Let $\mathbb{K}$ be a field.
A \emph{commutative differential graded algebra } $\left( A,\d \right)$
(CDGA short) over
$\mathbb{K}$
is a graded algebra $A = \bigoplus_{k\ge 0} A_k$ over $\mathbb{K}$ such that for all $x\in A_k$
and $y\in A_l$ we have
$$
x y = \left( -1 \right)^{kl} yx
$$
and $\d\colon A_k \to A_{k+1}$ is a differential, i.e. $d^2 =0$.
An example of real commutative differential graded algebra is given by the de Rham
complex $\left( \Omega^*\left( M \right), \d \right)$ of differential forms on a
smooth manifold $M$, with the multiplication given by the wedge product.

 A CDGA $\left( A,\d \right)$ is
 \emph{directly quasi-isomorphic} to  a CDGA $\left( B,\d
\right)$ if there is a
homomorphism of graded algebras $f\colon A\to B$ such that
$$
H^k\left( f \right)\colon H^k\left( A \right) \to H^k\left( B \right)
$$
 are isomorphisms for all $k\ge 0$.
Two CDGAs $\left( A,\d \right)$ and $\left( B,\d
\right)$ are \emph{quasi-isomorphic} if there is a chain of CDGAs $A=A_0$, $A_1$, \dots, $A_r = B$, such that either $A_j$ is directly
quasi-isomorphic to $A_{j+1}$ or
$A_{j+1}$ is directly quasi-isomorphic to  $A_j$.

Note that from the theory of Sullivan of minimal models it follows that the
chain in the above definition can be always chosen to have $r\le 2$.

For every CDGA $\left( A,\d \right)$ the cohomology algebra $H\left( A
\right)$ can be considered as a CDGA with the zero differential. A CDGA
$\left( A,\d \right)$ is called \emph{formal} if $\left( A,\d \right)$ and
$\left( H\left( A \right), 0 \right)$ are quasi-isomorphic.
A manifold $M$ is called \emph{formal} if $\left( \Omega^*\left( M
\right),\d  \right)$ is a formal CDGA, that is if the cohomology ring of
$M$ is quasi-isomorphic to the de Rham algebra of $M$.


In~\cite{sullivan77}
Sullivan constructed for every connected topological space $X$
a CDGA $\left(A\left( X \right), \d\right)$  over $\mathbb{Q}$
such that if $X$ is a manifold then $\left(A\left( X
\right)\otimes_{\mathbb{Q}} \mathbb{R},\d \right)$ is quasi-isomorphic to $\left(
\Omega^*\left( X \right), \d
\right)$.
He also proved that $\left( A\left( X \right),\d \right)$
is formal if and only if $\left( \Omega^*\left( X \right), \d \right)$ is
formal.

We say that two topological spaces $X$ and $Y$ have the same \emph{rational
homotopy type} if there is a finite chain of maps
$$
X \to Y_1 \leftarrow Y_2 \to Y_3 \leftarrow \dots \to Y
$$
such that the induced maps in rational cohomology are isomorphisms.

Let $G$ be a group and $H_1$, $H_2$ two subgroups of $G$. Then $\left[H_1, H_2
\right]$ is the set of elements $h_1 h_2 h_1^{-1}h_2^{-1}$ with $h_1\in H_1$,
$h_2\in H_2$. The
\emph{lower central series} of $G$ is defined to be the chain of subgroups  in
$G$
$$
G \supset G_1 \supset G_2 \supset G_3\supset \dots
$$
where $G_1:=\left[ G,G \right]$ and $G_k := \left[ G_k, G \right]$.
The group $G$ is called \emph{nilpotent} if there is an integer $n$ such that
$G_n$ is trivial. For example, abelian groups are nilpotent since
$\left[ G,G \right]$ is trivial in this case. On the other hand,
non-commutative simple groups are never nilpotent.

Denote by $\varepsilon \colon \mathbb{Z}G\to \mathbb{Z}$ the \emph{augmentation}
homomorphism, that is $\varepsilon \left( \sum_{g\in G} c_g g \right) =
\sum_{g\in G} c_g$. Define the augmentation ideal $I(G)$ of $\mathbb{Z}G$ to be the
kernel of $\varepsilon$. A left $G$-module $N$ is called \emph{nilpotent} if
there is $n\in \mathbb{N}$ such that $I(G)^n N =0$.
This notion can also be explained differently. Define the submodules $N_k\subset
N$ by
\begin{align*}
    N_0 &:= N, & N_k := \left\langle x - g x \,\middle|\, x \in N_{k-1},\
    g\in G \right\rangle.
\end{align*}
Then $N$ is nilpotent if and only if there is $n\in \mathbb{N}$ such that
$N_n = 0$.

Recall, that for every topological space $X$ the fundamental group $\pi_1(X)$
acts on the higher homotopy groups $\pi_k\left( X \right)$, $k\ge 2$. Thus we
can consider $\pi_k\left( X \right)$ as $\pi_1\left( X \right)$-module for every
$k\ge 2$. The topological space $X$ is called a \emph{nilpotent topological
space} if $\pi_1\left(
X \right)$ is a nilpotent group and the modules $\pi_k\left( X \right)$ are
nilpotent $\pi_1\left( X \right)$-modules.


The following theorem is due to Sullivan.
\begin{theorem}
    Two nilpotent topological spaces $X$ and $Y$ with finite Betti numbers have the same rational homotopy type if and only if
    $\left( A\left( X \right), \d \right)$ and $\left( A\left( Y
    \right), \d \right)$ are quasi-isomorphic CDGAs over $\mathbb{Q}$.
\end{theorem}
Thus in the case of a nilpotent topological space  $X$ the algebra
$\left( A\left( X \right), \d \right)$ contains full information about the  rational
homotopy type  of $X$. In particular, the groups $\pi_k\left( X \right)\otimes
\Qq$ can
be computed from any CDGA quasi-isomorphic $\left( A\left( X \right), \d \right)$.
If $M$ is a nilpotent manifold then the groups
$\pi_k\left( M \right)\otimes \Rr$ can be computed from any CDGA
quasi-isomorphic to $\left( A\left( M \right)\otimes \Rr, d \right)$ or to
$\left( \Omega^*(M),d \right)$.
In particular, if $M$ is a formal manifold, then the groups $\pi_k\left( M \right)\otimes \Rr$ can be
computed in purely formal way from $H^*\left( M \right)$.

In~\cite{deligne75}, Deligne et al. proved that every compact K\"ahler manifold is
formal. Using this, Chinea, de Le\'on and Marrero showed
in~\cite{chinea-deleon-marrero93} that
every compact coK\"ahler manifold is formal as well. This result is most
useful if the fundamental group of the manifold in question is nilpotent. Note
that in contrast to the case of K\"ahler manifolds the fundamental group of
a coK\"ahler manifold cannot be trivial.

In~\cite{bazzoni12} Bazzoni, Fern\'andez and Mu\~noz presented a way to
construct non-formal compact cosymplectic manifolds. The idea is as follows. Let
$S$ be a symplectic manifold and $f \colon S\to S$ a symplectomorphism
such that for some $p>0$ the eigenvalue $\lambda =1$ of
$$
H^p\left( f \right) \colon H^p\left( S \right) \to H^p\left( S \right)
$$
has multiplicity two. Then by \cite[Theorem~13]{bazzoni12} and by Theorem~\ref{mappingtorus}  the mapping torus $S_f$ is a  non-formal cosymplectic manifold. Using
this result Bazzoni, Fern\'andez and Mu\~noz show that for every pair $\left( m=2n+1, b \right)$ such that either $m=3$ and $b\ge 2$ or $m\ge 5$ and $b\ge
1$, there is a non-formal compact cosymplectic manifold $M$  of dimension $m$ such that
the first Betti number of $M$ is $b$.

\subsection{Cosymplectic $3$-structures}
When a smooth manifold $M$ is endowed with three distinct almost contact structures $(\phi_1,\xi_1,\eta_1)$, $(\phi_2,\xi_2,\eta_2)$,
$(\phi_3,\xi_3,\eta_3)$ related by the following identities
\begin{equation} \label{quaternionic}
\begin{split}
\phi_\gamma=\phi_{\alpha}\phi_{\beta}-\eta_{\beta}\otimes\xi_{\alpha}=-\phi_{\beta}\phi_{\alpha}+\eta_{\alpha}\otimes\xi_{\beta},\quad\\
\xi_{\gamma}=\phi_{\alpha}\xi_{\beta}=-\phi_{\beta}\xi_{\alpha}, \ \ \eta_{\gamma}= \eta_{\alpha}\circ\phi_{\beta}=-\eta_{\beta}\circ\phi_{\alpha},
\end{split}
\end{equation}
for any even permutation $(\alpha,\beta,\gamma)$ of the set
$\left\{1,2,3\right\}$, we say that $M$ is endowed with an
\emph{almost contact 3-structure}. In this case $M$ has dimension of
the form $4n+3$. This notion was introduced independently by Kuo
(\cite{kuo}) and Udriste (\cite{udriste}) at the end of the 60s and
then it was studied by several authors, especially from the Japanese
school of Riemannian geometry. In particular, Kuo proved that one
can always find a Riemannian metric $g$ which is compatible with
each almost contact structure. If we fix one, we speak of an
\emph{almost contact metric 3-structure}.

Any  smooth manifold endowed with an almost contact metric $3$-structure  carries two orthogonal distributions, not necessarily integrable: the
\emph{Reeb distribution} ${\mathcal V}:=\textrm{span}\{\xi_1,\xi_2,\xi_3\}$ and the \emph{horizontal distribution} ${\mathcal
H}:=\ker(\eta_1)\cap\ker(\eta_2)\cap\ker(\eta_3)={\mathcal V}^{\perp}$.

The most famous class of almost contact metric $3$-structures is given by  those for which each structure is Sasakian. They are called \emph{$3$-Sasakian
structures} (or  Sasakian $3$-structures). For more details on $3$-Sasakian manifolds see \cite{galickibook}.

When each structure $(\phi_{\alpha},\xi_{\alpha},\eta_{\alpha},g)$ is (almost) coK\"{a}hler, $M$ is called an \emph{(almost) $3$-cosymplectic manifold}.
However, it has been proved recently (see \cite[Theorem 4.13]{pastore})  that these two notions are the same, i.e. every almost $3$-cosymplectic manifold
is normal and thus $3$-cosymplectic.

\begin{remark}
Actually, the name \emph{$3$-cosymplectic} is not coherent with the
terminology that we used in the rest of this survey. A more
appropriate name should have been ``3-coK\"{a}hler''. However in
literature the term $3$-cosymplectic is well-established and, so
far, we did not find any other name for these structures (except
``hypercosymplectic $3$-structure'' in \cite{song-kim-tripathi2003}
and \cite{kim-choi-tripathi2006}). Thus in this section we shall
conform to the terminology currently used in literature.
\end{remark}

\begin{remark}
We notice that just as in the case of a single structure, the $3$-Sasakian and the $3$-cosymplectic manifolds represent the two extremal cases of the
larger class of 3-quasi-Sasakian manifolds (cf. \cite{cappellettidenicoladileo1}, \cite{cappellettidenicoladileo2}).
\end{remark}

In any $3$-cosymplectic manifold the Reeb vector fields $\xi_1$, $\xi_2$, $\xi_3$ are all $\nabla$-parallel. In particular, for any
$\alpha,\beta\in\left\{1,2,3\right\}$, we have $[\xi_\alpha,\xi_\beta]=\nabla_{\xi_\alpha}\xi_\beta-\nabla_{\xi_\alpha}\xi_\beta=0$. Thus the Reeb
distribution is integrable and defines a $3$-dimensional foliation ${\mathcal F}_{3}$ of $M$. As it was proved in \cite{cappellettidenicola}, ${\mathcal
F}_{3}$ is a Riemannian and transversely hyper-K\"{a}hler foliation with totally geodesic leaves. Using this, one proves that every 3-cosymplectic manifold
is Ricci-flat (\cite[Corollary 3.10]{cappellettidenicola}). Moreover, since $\d\eta_\alpha=0$, also the horizontal distribution $\mathcal H$ is integrable
and hence defines a Riemannian, totally geodesic foliation complementary to ${\mathcal F}_3$.

On the other hand, Ishihara proved that the Reeb distribution of a $3$-Sasakian manifold is a transversely quaternionic-K\"{a}hler foliation
(\cite{ishihara}).
\medskip
\begin{table}[ht]\centering\caption{\textsl{Transverse geometry}}
    \label{table}
\begin{tabular}{lcc|cc}
&\multicolumn{2}{c}{$\dim(M)=2n+1$}&\multicolumn{2}{c}{$\dim(M)=4n+3$}\\ \hline\hline
&Sasakian&coK\"{a}hler&$3$-Sasakian&$3$-cosymplectic\\
&$\downarrow$&$\downarrow$&$\downarrow$&$\downarrow$\\
&K\"{a}hler&K\"{a}hler&quaternionic-K\"{a}hler&hyper-K\"{a}hler\\
\end{tabular}
\end{table}
\medskip
The study of transverse geometry allows to understand the difference
between Sasakian and coK\"{a}hler structures, as illustrated in
Table~\ref{table}. In fact, while both
Sasakian and coK\"{a}hler manifolds are considered as a natural
odd-dimensional counterpart of K\"{a}hler manifolds, at the level of
$3$-structures their role differs: $3$-cosymplectic manifolds should
be considered as the most natural odd-dimensional analogue of
hyper-K\"{a}hler manifolds, whereas $3$-Sasakian structures
corresponds to quaternionic-K\"{a}hler structures.

\medskip

The standard example of a compact 3-cosymplectic manifold is given by the torus $\mathbb{T}^{4n+3}$ with the following structure (cf.
\cite[p.561]{martincabrera}). Let $\left\{\theta_{1},\ldots,\theta_{4n+3}\right\}$ be a basis of $1$-forms such that each $\theta_{i}$ is integral and
closed. Let us define a Riemannian metric $g$ on $\mathbb{T}^{4n+3}$ by
\begin{equation*}
g:=\sum_{i=1}^{4n+3}\theta_{i}\otimes\theta_{i}.
\end{equation*}
For each $\alpha\in\left\{1,2,3\right\}$  we define a tensor field $\phi_\alpha$ of type $(1,1)$ by
\begin{align*}
\phi_{\alpha}=\sum_{i=1}^{n}& \left( E_{\alpha n+i}\otimes\theta_{i}-E_{i}\otimes\theta_{\alpha+i}+E_{\gamma n+i}\otimes\theta_{\beta n +i}-E_{\beta n
+i}\otimes\theta_{\gamma n+i}\right) +E_{4n+\gamma}\otimes\theta_{4n+\beta}-E_{4n+\beta}\otimes\theta_{4n+\gamma}
\end{align*}
where $\left\{E_{1},\ldots,E_{4n+3}\right\}$ is the dual (orthonormal) basis of $\left\{\theta_{1},\ldots,\theta_{4n+3}\right\}$ and
$\left(\alpha,\beta,\gamma\right)$ is a cyclic permutation of $\left\{1,2,3\right\}$. Setting, for each $\alpha\in\left\{1,2,3\right\}$,
$\xi_{\alpha}:=E_{4n+\alpha}$ and $\eta_{\alpha}:=\theta_{4n+\alpha}$, one can easily check that the torus $\mathbb{T}^{4n+3}$ endowed with the structure
$(\phi_{\alpha},\xi_{\alpha},\eta_{\alpha},g)$ is 3-cosymplectic.

For the non-compact case, the standard example of 3-cosymplectic manifold is given by $\mathbb{R}^{4n+3}$ with the structure described in \cite[Theorem
4.4]{cappellettidenicola}.

Both the  examples  above are the global product of a
hyper-K\"{a}hler manifold with a $3$-dimensional abelian Lie group.
Such a property always holds locally (see \cite[Proposition
7.1]{cappellettidenicolayudin}). Thus it makes sense to ask whether there
are examples of $3$-cosymplectic manifolds which are not the global
product of a hyper-K\"{a}hler manifold with a $3$-dimensional
abelian Lie group. As shown in \cite{cappellettidenicolayudin}, the
answer to this question is affirmative and now we illustrate a
procedure for constructing such examples.  Let $(N,J_{\alpha},G)$ be
a compact hyper-K\"{a}hler manifold of dimension $4n$ and $f$ a
hyper-K\"{a}hler isometry on it. We define an action $\varphi$ of
$\mathbb{Z}^3$ on $N\times\mathbb{R}^{3}$ by
\begin{equation*}
\varphi((k_1,k_2,k_3),(x,t_1,t_2,t_3)) = (f^{k_1+k_2+k_3}(x), t_{1}+k_{1},t_{2}+k_{2},t_{3}+k_{3}).
\end{equation*}
We then define on the orbit space
$M_{f}:=(N\times\mathbb{R}^{3})/\mathbb{Z}^{3}$ a $3$-cosymplectic
structure in the following way. Let us consider the vector fields
$\xi_\alpha:=\frac{\partial}{\partial t_\alpha}$ and the $1$-forms
$\eta_\alpha:=d t_\alpha$ on $N\times \mathbb{R}^{3}$. Next we
define, for each $\alpha\in\left\{1,2,3\right\}$, a tensor field
$\phi_\alpha$ on $N\times\mathbb{R}^{3}$ by putting
$\phi_{\alpha}X:=J_{\alpha}X$ for any $X\in\Gamma(TN)$ and
$\phi_{\alpha}\xi_{\alpha}:=0$,
$\phi_{\alpha}\xi_{\beta}:=\epsilon_{\alpha\beta\gamma}\xi_{\gamma}$,
where $\epsilon_{\alpha\beta\gamma}$ denotes the sign of the
permutation $(\alpha,\beta,\gamma)$ of $\left\{1,2,3\right\}$. Then
$(\phi_\alpha,\xi_\alpha,\eta_\alpha,g)$,  $g$ denoting the product
metric, defines a $3$-cosymplectic structure on
$N\times\mathbb{R}^{3}$. Being invariant under the action $\varphi$,
$(\phi_\alpha,\xi_\alpha,\eta_\alpha,g)$ descends to a
$3$-cosymplectic structure on $M_{f}$, which we will denote by the
same symbol. By using this general procedure in
\cite{cappellettidenicolayudin} a   non-trivial example of compact
$3$-cosymplectic manifold is constructed. In fact we consider the
hyper-K\"{a}hler manifold $\mathbb{T}^{4}=\mathbb{H}/\mathbb{Z}^4$
and the hyper-K\"{a}hler isometry $f$ induced by the multiplication by
$\mathbf{i}$ (the same can be made by using the multiplication by
$\mathbf{j}$ and $\mathbf{k}$). Then the $7$-dimensional manifold
$M_{f}:=(\mathbb{T}^{4}\times\mathbb{R}^{3})/\mathbb{Z}^{3}$,
endowed with the geometric structure described above, is a compact
$3$-cosymplectic manifold which is not the global product of a
compact $4$-dimensional hyper-K\"{a}hler manifold $K$ with the flat
torus. Indeed one has only two possibilities for a compact
$4$-dimensional hyper-K\"{a}hler manifold: either
$K\cong\mathbb{T}^{4}$ or it is a complex K3-surface. In the first
case $b_{2}(K\times\mathbb{T}^{3})=21$, in the second
$b_{2}(K\times\mathbb{T}^{3})=25$. However, in
\cite{cappellettidenicolayudin} it was proved that
$b_{2}(M^{7}_{f})<21$.

\medskip

We close the section by collecting some results on the topology of compact $3$-cosymplectic manifolds. Every $3$-cosymplectic manifold is, in particular,
coK\"{a}hler, so that all the results in $\S$ \ref{topology} hold. However the rigidity given by the relations \eqref{quaternionic} connecting the almost
contact structures forces the topology of a $3$-cosymplectic manifold to satisfy the following more restrictive properties.

Let $b_{k}^{h}$ denote the $k$-th basic Betti number with respect to the foliation ${\mathcal F}_3$, that is the dimension of the vector space
\begin{equation*}
\Omega^{k}_{HB}(M)=\left\{\omega\in\Omega^{k}_{H}(M) | i_{\xi_\alpha}\omega=0 \ \hbox{ for any }\alpha=1,2,3 \right\}.
\end{equation*}

\begin{theorem}[\cite{cappellettidenicolayudin}]
Let $M$ be a compact $3$-cosymplectic manifold of dimension $4n+3$. Then
the $k$-th Betti numbers of
$M$ are given by
\begin{equation}\label{bettiformula}
        b_{k} =b_{k}^{h} + 3b_{k-1}^{h} + 3b_{k-2}^{h} + b_{k-3}^{h},
    \end{equation}
where we assume that $b_p^h=0$ for $p<0$.
\end{theorem}

By using \eqref{bettiformula} one proves the following theorem.

\begin{theorem}[\cite{cappellettidenicolayudin}]
Let $M$ be a compact $3$-cosymplectic manifold of dimension $4n+3$. Then the following statements hold.
\begin{enumerate}
\item[(i)] Any odd-dimensional basic Betti number $b_{2k+1}^{h}$ is divisible by $4$.
\item[(ii)] For any integer $k$, $b_{k-1} + b_{k}$ is divisible by $4$.
\end{enumerate}
\end{theorem}

In \cite{cappellettidenicolayudin}, we also proved certain inequalities that the Betti numbers have to satisfy. First notice that, as the $1$-forms $\eta_{1}$, $\eta_{2}$, $\eta_{3}$, the $2$-forms $\Phi_{1}$, $\Phi_{2}$, $\Phi_{3}$, $\eta_{i}\wedge\eta_{j}$,   and the $3$-forms
$\eta_{1}\wedge\eta_{2}\wedge\eta_{3}$, $\eta_{i}\wedge\Phi_{j}$,  are harmonic and linearly independent, we have $b_{1}\geq 3$,
$b_{2}\geq 6$, $b_{3}\geq 10$.  More generally, we have the following formula.

\begin{theorem}[\cite{cappellettidenicolayudin}]\label{binomiale}
Let $M$ be a compact $3$-cosymplectic manifold of dimension $4n+3$. Then, for any integer $k\in\left\{0,\ldots,2n+1\right\}$,
\begin{equation}\label{bin1}
b_{k}\geq \binom{k+2}{2}.
\end{equation}
\end{theorem}

Theorem \ref{binomiale} should be compared with the corresponding
even-dimensional result. Indeed, due to Wakakuwa (\cite{wakakuwa}), in any compact hyper-K\"{a}hler manifold
the even Betti number $b_{2k}$ ($0\leq k \leq 2n$) satisfies
\begin{equation}\label{bin2}
b_{2k}\geq \binom{k+2}{2}.
\end{equation}
Comparing \eqref{bin1} with \eqref{bin2}, we clearly note that in the $3$-cosymplectic case we have a much stronger condition.

We conclude the section by mentioning that the space
$\Omega^{k}_{HB}(M)$ of harmonic basic $k$-forms  admits an action
of the Lie algebra $\mathfrak{so}(4,1)$. For each
$\alpha\in\left\{1,2,3\right\}$, let us define the $2$-form
\begin{equation*}
\Xi_\alpha:=\frac{1}{2}(\Phi_{\alpha}+2\eta_{\beta}\wedge\eta_{\gamma})
\end{equation*}
where $(\alpha,\beta,\gamma)$ is an  even permutation  of $\{1,2,3\}$. Next we define the operators
\begin{gather*}
L_{\alpha}:\Omega^{k}(M)\longrightarrow\Omega^{k+2}(M), \ \ \  L_{\alpha}\omega:=\Xi_{\alpha}\wedge\omega\\
\Lambda_{\alpha}:\Omega^{k+2}(M)\longrightarrow\Omega^{k}(M), \ \ \ \Lambda_{\alpha}:=\star L_{\alpha} \star
\end{gather*}
One can prove that $L_\alpha$ and $\Lambda_\alpha$  induce
endomorphisms  of $\Omega^{\ast}_{HB}(M)$ (which we shall denote by the same symbol). Then, by \cite[Proposition 4.3]{cappellettidenicolayudin} one has that, on
$\Omega^{\ast}_{HB}$, $[L_{\alpha},\Lambda_{\alpha}]=-H$, where $H:\Omega^{k}_{HB}(M)\longrightarrow\Omega^{k}_{HB}(M)$ is the operator defined by
$H\omega=(2n-k)\omega$. Moreover for each $\alpha\in\{1,2,3\}$ we define another operator $K_\alpha$ on $\Omega^{k}_{HB}(M)$ by
$K_{\alpha}:=[L_\beta,\Lambda_\gamma]$, where $(\alpha,\beta,\gamma)$ is an even permutation of $\{1,2,3\}$. Then we have the following result.

\begin{theorem}[\cite{cappellettidenicolayudin}]
The linear span $\mathfrak{g}$ of the operators $H$, $L_\alpha$, $\Lambda_\alpha$, $K_\alpha$, $\alpha\in\left\{1,2,3\right\}$, is a Lie algebra isomorphic
to $\mathfrak{so}(4,1)$. Consequently $\Omega^{\ast}_{HB}(M)$ is an $\mathfrak{so}(4,1)$-module.
\end{theorem}

\section{Further topics}

There are many other topics related to cosymplectic / coK\"{a}hler geometry. Unfortunately, due  to the lack of time and space,  we can only give a very
brief outline.

\subsection{Submanifolds of coK\"{a}hler manifolds}

In almost contact Riemannian geometry there have always been a great interest  toward the theory of submanifolds. The most important class of submanifolds
of an almost contact metric manifold $(M,\phi,\xi,\eta,g)$ is given by \emph{invariant submanifolds}.

A submanifold $M'$ of $M$ is said to be invariant if $\xi$ is tangent to the submanifold and, for each $x\in M'$, $\phi (T_{x}M') \subset T_{x}M'$. Then
$M'$ inherits an almost contact metric structure $(\phi',\xi',\eta',g')$ from the ambient space by restriction and $(\phi',\xi',\eta',g')$ is (almost)
coK\"{a}hler provided that $(\phi,\xi,\eta,g)$ is (almost) coK\"{a}hler (\cite{ludden70}). An important result that invariant submanifolds of almost
coK\"{a}hler manifolds share with contact Riemannian geometry (\cite{blairbook2010}, page 152) is the following, due to Endo.

\begin{theorem}[\cite{endo1985}]\label{invsubm1}
Any invariant submanifold of an almost coK\"{a}hler manifold is minimal.
\end{theorem}

Theorem \ref{invsubm1} suggests to investigate for conditions ensuring that an invariant submanifold of an (almost) coK\"{a}hler manifold is totally
geodesic. In \cite{ludden70} and \cite{goldberg71} Ludden and Goldberg, respectively, studied this problem for codimension $2$ submanifolds.

In \cite{endo1985, endo1989}, Endo obtained some partial results on invariant submanifolds that we collect in the following theorem.

\begin{theorem}\label{invsubm2}
Let $M'$ be an invariant submanifold of an almost coK\"{a}hler manifold $M$ with pointwise constant $\phi$-sectional curvature $c$. Let $2m+1$ be the
dimension of $M'$.
\begin{enumerate}
  \item[(i)] If $M'$ has constant $\phi'$-sectional curvature $c'$ then   $c' \leq c$, with equality holding if and only if $M'$ is totally geodesic.
  \item[(ii)]  The scalar curvature $s'$ of $M'$ satisfies the inequality $s'\leq m(m+1)c$. Moreover, the equality holds if and only if $M'$ is totally geodesic and coK\"{a}hler.
\end{enumerate}
\end{theorem}

Another remarkable class of submanifolds of an almost contact metric manifold is given \emph{anti-invariant submanifolds}. A submanifold of an almost
contact metric manifold is said to be anti-invariant if $\phi (T_{x}M') \subset (T_{x}M')^\perp$. A systematic study of anti-invariant submanifolds of
coK\"{a}hler manifolds was developed by Kim  \cite{kim83}. Among other result, the author found conditions ensuring that an anti-invariant submanifold of a
coK\"{a}hler manifold is conformally flat, or locally symmetric, or a totally geodesic submanifold.

We just mention also that the notion of invariant and anti-invariant submanifold were generalized in the class of \emph{CR-submanifold}, which was proposed
by Bejancu in the context of almost Hermitian geometry and then extended to the almost contact setting. Namely, a  submanifold $M'$ of an almost
coK\"{a}hler manifold $M$ is called CR-submanifold if $\xi$ is tangent to $M'$ and there exists an invariant distribution $D$ whose orthogonal
complement $D^\perp$ is anti-invariant, i.e. $TM'=D\oplus D^\perp \oplus \mathbb{R}\xi$ with $\phi D_x \subset D_x$ and $\phi D^\perp_x
\subset (T_x M')^\perp$. Further details can be found in \cite{bejancu-book86} and \cite{yano-kon-book83}.

\medskip

The study of the geometry of submanifolds of coK\"{a}hler manifolds is a current topic of research. There are several papers dealing with the above and
other classes of submanifolds of a coK\"{a}hler manifold (especially  coK\"{a}hler space forms). Here we mention the very recent paper
\cite{fetcu-rosenberg} of Fectu and Rosenberg. In that paper, they study the two-dimensional submanifolds with parallel mean curvature in a coK\"{a}hler
space form.

\subsection{Harmonic maps and coK\"{a}hler geometry}
Another interesting topic is the study of the interplay between coK\"{a}hler geometry with the theory of harmonic maps. Recall that given a smooth map $f:(M^m,g)\longrightarrow ({M'},g')$ between
Riemannian manifolds, one defines  the energy density of $f$ as the smooth function $e(f):M \longrightarrow [0,+\infty)$ given by
\begin{equation*}
e(f)(x)=\frac{1}{2}\left\|f_{{\ast} x}\right\|^2=\frac{1}{2}\sum_{i=1}^{m}g'\left(f_{{\ast} x}(e_i),f_{{\ast} x}(e_i)\right),
\end{equation*}
for any $x\in M$, where $\left\{e_{1},\ldots,e_{m}\right\}$ is any local orthonormal basis of $T_{x}M$. If $M$ is compact, the \emph{energy} of $f$ is defined by
\begin{equation*}
E(f)=\int_{M}e(f)\nu_{g}
\end{equation*}
where $\nu_{g}$ is the volume measure associated with the metric $g$ on $M$. Then $f$ will be said to be a \emph{harmonic map} if it is a critical point
of the energy functional $E$ on the set of all maps between $(M,g)$ and $(M',g')$.

In particular it is of interest to study the harmonicity of the maps to, from and between (almost) coK\"{a}hler manifolds, especially in the case when one of the two spaces is K\"{a}hler. We refer the reader
to the papers of Boeckx and Gherghe (\cite{boeckx2004}), Chinea (\cite{chinea2010}), Gherghe (\cite{gherghe1999}, \cite{gherghe2010}) and Fectu (\cite{fectu2004}).

\subsection{Generalizations of cosymplectic and coK\"{a}hler manifolds}
Several generalizations of cosymplectic and coK\"{a}hler manifolds have been considered. We mention just a few of them, recalling the definitions and
referring the interested reader to the references below.

The first generalization that we recall is given by \emph{locally conformal cosymplectic manifolds}, introduced by Chinea,  de Le\'{o}n,  Marrero in
\cite{chinea-deleon-marrero91} and \emph{locally conformal coK\"{a}hler manifolds}, defined by Olszak  in \cite{olszak89}. An almost cosymplectic manifold
$(M,\eta, \omega)$ is said to be \emph{locally conformal cosymplectic} (l.c.c.), if there exist an open covering $\{U_i\}_{i\in\ I}$, and a family of
functions $\sigma : U_i \longrightarrow \mathbb{R}$ such that  $\eta_i=e^{\sigma_{i}}\eta|_{U_i}$ and $\omega_i=e^{2\sigma_{i}}\omega|_{U_i}$  are closed,
i.e. $(\eta_i,\omega_i)$ is a cosymplectic structure on $U_i$. Then one proves that the local $1$-forms $d\sigma_i$ glue up to a closed $1$-form $\theta$,
called the \emph{Lee form}, which satisfies
\begin{equation}\label{lcc}
\d\eta=\eta\wedge\theta, \ \ \  \d\omega = - 2 \omega \wedge \theta.
\end{equation}
Clearly, if $(M,\omega,\eta)$ is cosymplectic, then it is l.c.c. with zero Lee form. Analogously one gives the definition of locally conformal  coK\"{a}hler manifolds. Let $(M,\phi,\xi,\eta,g)$ be an almost contact metric  manifold. Then
$M$ is said to be \emph{locally conformal coK\"{a}hler} (l.c.cK.) if there exists an open covering $\{U_i\}_{i\in\ I}$ of $M$ and a family $\{\sigma_i\}_{i\in
I}$  of real-valued smooth functions on each $U_i$ such that $(U_i , \phi_i , \xi_i , \eta_i , g_i)$ is a coK\"{a}hler manifold, where
\begin{equation*}
\phi_i = \phi|_{U_i}, \ \ \ \xi_i=e^{-\sigma_i}\xi|_{U_i}, \ \ \ \eta_i=e^{\sigma_i}\eta|_{U_i}, \ \ \ g_i=e^{2\sigma_i}g|_{U_i}
\end{equation*}
From the definition it follows that any locally conformal coK\"{a}hler manifold
is normal. Also, one proves that for any $i, j \in I$, $i \neq j$, with $U_i
\cap U_j \neq \emptyset$, one has $d\sigma_i = d\sigma_j$ on $U_i \cap U_j$, so that the  $1$-forms $d\sigma_i$ glue up to a global closed $1$-form
$\theta$. Moreover, the Levi-Civita connections $\nabla^i$ of $(U_i , g_i)$ glue up to a globally defined torsion-free linear connection $D$ on $M$ given
by
\begin{equation*}
D_{X}Y = \nabla_{X}Y - \frac{1}{2}\left(\theta(X)Y + \theta(Y)X - g(X,Y)B\right),
\end{equation*}
where $B = \theta^\sharp$  and $\nabla$ is the Levi-Civita
connection of $(M, g)$. It follows that $\d\eta=\eta\wedge\theta$
and $\d\Phi=-2\Phi\wedge\theta$, $\Phi$ being the fundamental
$2$-form of the almost contact metric manifold
$(M,\phi,\xi,\eta,g)$. For more details we refer the reader to the
aforementioned papers or also the monograph
\cite{dragomir-ornea-book}.

\medskip

We mention also the \emph{para-coK\"{a}hler manifolds}, studied, among others, by Dacko and Olszak (cf. \cite{dacko04}, \cite{dacko-olszak07}). An
\emph{almost paracontact structure} (cf. \cite{kaneyuki85}) on a $(2n+1)$-dimensional smooth manifold $M$ is given by a $(1,1)$-tensor field $\tilde\phi$, a vector
field $\xi$ and a $1$-form $\eta$ satisfying the following conditions
\begin{enumerate}
  \item[(i)] $\eta(\xi)=1$, \ $\tilde\phi^2=I-\eta\otimes\xi$,
  \item[(ii)] the eigendistributions ${\mathcal D}^+$ and ${\mathcal D}^-$ of $\tilde\phi$ corresponding to the eigenvalues $1$ and $-1$, respectively, have equal dimension $n$.
\end{enumerate}
An almost paracontact structure $(\tilde\phi,\xi,\eta)$ is said to be \emph{normal} if the tensor field
$N_{\tilde\phi}:=[\tilde\phi,\tilde\phi]-2\d\eta\otimes\xi$ vanishes identically.  Equivalently, an almost paracontact manifold  is normal if and only if
the distributions ${\mathcal D}^{+}$ and ${\mathcal D}^{-}$  are integrable and  $\xi$ is a foliated vector field with respect to both foliations
(\cite{cappelletti09}). If an almost paracontact manifold is endowed with  a semi-Riemannian metric $\tilde{g}$ such that
\begin{equation}\label{compatibile}
\tilde{g}(\tilde\phi X,\tilde\phi Y)=-\tilde{g}(X,Y)+\eta(X)\eta(Y)
\end{equation}
for all  $X,Y\in\Gamma(TM)$, then $(M,\tilde{\phi},\xi,\eta,\tilde{g})$ is called an \emph{almost paracontact metric manifold}. Notice that any such
semi-Riemannian metric is necessarily of signature $(n,n+1)$, where the first
index denotes the number of negative signs, and the above condition (ii) of the definition of almost paracontact structures is
automatically satisfied. Moreover, as in the almost contact case, from \eqref{compatibile} it follows easily that $\eta=\tilde{g}(\cdot,\xi)$ and
$\tilde{g}(\cdot,\tilde\phi\cdot)=-\tilde{g}(\tilde\phi\cdot,\cdot)$. Hence one defines the \emph{fundamental $2$-form} of the almost paracontact metric
manifold by $\tilde\Phi(X,Y)=\tilde{g}(X,\tilde\phi Y)$. Now, if $\eta$ and $\tilde\Phi$ are closed the structure is said to be \emph{almost
para-coK\"{a}hler}, and we point out that  $(\eta,\tilde\Phi)$ is then a cosymplectic structure. Finally, a normal almost para-coK\"{a}hler manifold is said to be  \emph{para-coK\"{a}hler}. Para-coK\"{a}hler geometry has a long history, having its origins in the three papers \cite{buzon62-I}, \cite{buzon62-II}, \cite{buzon64} of Bouzon in the early 60's (even before than the Blair's definition of coK\"{a}hler manifolds).
The theory has some aspects similar to coK\"{a}hler geometry even though there are also some relevant differences, for
instance due to the non-positive definiteness of the metric.

\medskip

Furthermore, another  generalization of cosymplectic manifolds  was proposed recently by de Le\'{o}n,  Merino, Oubi\~{n}a, Rodrigues and Salgado in
\cite{deleon-merino-oubina-rodrigues-salgado98}. A \emph{$k$-cosymplectic manifold} is  a smooth manifold $M$ of dimension $(k + 1)n + k$ endowed with  family $(\eta_{\alpha}, \omega_{\alpha}, {\mathcal F})_{\alpha\in\left\{1,\ldots,k\right\}}$, where each $\eta_\alpha$ is a closed $1$-form, each $\omega_\alpha$ is a closed $2$-form and $\mathcal F$ is a $nk$-dimensional foliation on $M$, related by the following conditions
\begin{enumerate}
\item[(i)] $\eta_{1}\wedge\cdots\wedge\eta_{k}\neq 0$,
\item[(ii)] $\eta_\alpha(X)=0$ and $\omega_\alpha(X,X')=0$ \ for any $X,X'\in\Gamma(T{\mathcal F})$ and for each $\alpha\in\left\{1,\ldots,k\right\}$,
\item[(iii)] $\left(\bigcap\limits_{\alpha=1}^{k}\ker(\eta_\alpha)\right) \cap \left(\bigcap\limits_{\alpha=1}^{k}\ker(\omega_\alpha)\right) = \left\{0\right\}$ and
$\dim\left(\bigcap\limits_{\alpha=1}^{k}\ker(\omega_\alpha)\right)=k$.
\end{enumerate}
Then on $M$ there are defined   $k$ vector fields, $\xi_1,\ldots,\xi_k$, called \emph{Reeb vector fields}, such that
\begin{equation*}
i_{\xi_\alpha}\eta_{\beta}=\delta_{\alpha\beta}, \ \ \ i_{\xi_\alpha}\omega_{\beta}=0
\end{equation*}
for any $\alpha,\beta\in\left\{1,\ldots,k\right\}$. The $k$-cosymplectic manifolds may be described locally in terms of Darboux coordinates (\cite{deleon-merino-oubina-rodrigues-salgado98}). In fact about any
point of the manifold there are coordinates $(x_{\alpha},y_{i},z_{\alpha i})$, $1\leq\alpha\leq k$, $1\leq i\leq n$, such that
\begin{equation*}
\eta_{\alpha}=dx_{\alpha}, \ \  \omega_{\alpha}=\sum_{i=1}^{n}dy_{i}\wedge dz_{\alpha i}, \  \ T{\mathcal F}=\textrm{span}\left\{\frac{\partial}{\partial
z_{1 1}},\ldots,\frac{\partial}{\partial z_{k 1}},\ldots,
\frac{\partial}{\partial z_{1 n}},\ldots,\frac{\partial}{\partial z_{k
n}}\right\}.
\end{equation*}
The $k$-cosymplectic structures were introduced in order to provide a suitable geometric setting for the description of classical field theories. The reader can find more
information on the state of the art of the theory for instance in  \cite{munoz-salgado-vilarino2010} or   \cite{rey-romanroy-salgado-vilarino2012}.

\medskip

We conclude the subsection by mentioning the very recent concepts of ``contact-symplectic pairs'' and  ``generalized almost contact structure''. A
\emph{contact-symplectic pair of type $(h,k)$ }on a smooth manifold $M$ of dimension $2h+2k+1$ is given by a $1$-form $\eta$ and a closed $2$-form $\omega$
such that ${d\eta}^{h+1}=0$, $\omega^{k+1}=0$ and $\eta\wedge d\eta^{h}\wedge \omega^{k}$ is a volume form, for some integers $h$ and $k$. For $k=0$ one
retrieves the definition of contact structures and for $h=0$ that one of cosymplectic structures. This notion was introduced by Bande in \cite{bande03},
where he studies the main properties of such structures and finds examples on nilpotent Lie groups, nilmanifolds and principal torus bundles.

Next, the theory of generalized almost contact manifolds is the attempt to extend to odd dimensions the very important theory of generalized complex
structures, founded by Hitchin and Gualtieri (cf. \cite{hitchin}, \cite{gualtieri}). With this regard we mention the works of Iglesias-Ponte and Wade
(\cite{iglesiasiasponte-wade-2005}) , Vaisman (\cite{vaisman2008}), Poon and Wade (\cite{poon-wade2010}, \cite{poon-wade2011}). In particular we recall the
approach proposed by Vaisman and by Poon and Wade. According to their definition, a \emph{generalized almost contact structure} on a smooth manifold $M$
consists of a bundle endomorphism $\phi:TM\oplus T^{\ast}M \longrightarrow TM\oplus T^{\ast}M$ and section $\xi+\eta$ of $TM\oplus T^{\ast}M$ such that
\begin{gather*}
\phi + \phi^\ast = 0, \ \ \  \phi^2 = -I + \xi\odot\eta,  \ \ \ \eta(\xi)=1, \ \ \ \phi(\xi)=\phi(\eta)=0.
\end{gather*}
Here the adjoint $\phi^\ast$ is with respect to the natural pairing of $TM\oplus T^{\ast}M$ defined by
\begin{equation*}
\langle X+\alpha, Y+\beta\rangle = \frac{1}{2}\left(i_{X}\beta+i_{Y}\alpha\right),
\end{equation*}
and $\xi\odot\eta$ is the endomorphism defined by
$\xi\odot\eta(X+\alpha)=\eta(X)\xi  + \alpha(\xi)\eta $. This notion includes as particular cases both contact and almost cosymplectic structures. Moreover, Poon and Wade discuss possible  notions of normality / integrability.

\subsection{Almost coK\"{a}hler manifolds satisfying further conditions}
Another important topic in co-K\"{a}hler geometry is given by the study of some classes of almost coK\"{a}hler manifolds satisfying certain additional
conditions (for instance concerning the curvature tensor). More generally, one considers some sub-classes of almost contact metric manifolds which satisfy
conditions closely related to the coK\"{a}hler manifolds. Probably the most important example is given by the \emph{nearly coK\"{a}hler manifolds}. An
almost contact metric manifold $(M,\phi,\xi,\eta,g)$ is called nearly coK\"{a}hler  if the structure tensor $\phi$ and the $1$-form $\eta$ are Killing,
that is
\begin{equation}\label{nearly}
(\nabla_{X}\phi)Y + (\nabla_{Y}\phi)X=0, \ \ \ (\nabla_{X}\eta)(Y) + (\nabla_{Y}\eta)(X)=0
\end{equation}
for any $X,Y\in\Gamma(TM)$. The second condition in \eqref{nearly}, which is
equivalent to requiring that $\xi$ is Killing, is usually included in the
definition, but however it is a consequence of the first one (cf. \cite[Proposition 6.1]{blairbook2010}). The name ``nearly coK\"{a}hler'' is due to the
property that a normal nearly coK\"{a}hler manifold is necessarily coK\"{a}hler (\cite{blair71}). Moreover, it is known that any $3$-dimensional nearly
coK\"{a}hler manifold is coK\"{a}hler (\cite{jun-kim-kim-1994}). A well-known  nearly coK\"{a}hler manifold is $S^5$ endowed with the almost contact metric
structure induced by the almost Hermitian structure of $S^6$ defined by means of the cross product of the imaginary part of the octonions (see, for
more details, \cite{blairbook2010} page 63 and 102). This gives an example of non-normal almost contact metric structure on an odd-dimensional sphere. \
Nearly coK\"{a}hler manifolds, in particular their curvature and topological properties, have been studied by several authors (see e.g.
\cite{kiritchenko82}, \cite{banaru2002},  \cite{endo2005}, \cite{fueki-endo-2005}, \cite{endo2012}).

\medskip

Another interesting topic concerns the study of almost contact metric manifolds admitting a Weyl structure.  Recall that  a \emph{Weyl structure} on a
manifold $M$ of dimension $m \geq 3$ is defined by a pair $W = ([g],D^{[g]})$, where $[g]$ is a conformal class of Riemannian metrics and $D^{[g]}$ is the unique
torsion-free connection, called \emph{Weyl connection}, satisfying
\begin{equation}\label{weil2}
D^{[g]}g = -2\theta\otimes g.
\end{equation}
for some $1$-form $\theta$. The Weyl structure is said to be \emph{closed} if $\d\theta=0$  and  \emph{Einstein-Weyl} if there exists a smooth function
$\Lambda$ on $M$ such that
\begin{equation}\label{weil1}
\textrm{Ric}^{D^{[g]}}(X, Y) + \textrm{Ric}^{D^{[g]}}(Y,X) = \Lambda g(X, Y),
\end{equation}
where $\textrm{Ric}^{D^{[g]}}$ denotes the Ricci tensor with respect
to the connection $D^{[g]}$. Since the condition \eqref{weil2} is
invariant under Weyl transformations $g'=e^{2f}g$,
$\theta'=\theta+\d f$, with $f\in C^{\infty}(M)$, one sometimes abuses
the terminology by choosing a Riemannian metric in $[g]$ and
referring to the pair $W = (g, \theta)$ as a Weyl structure. \ The
existence of Einstein-Weyl structures on almost contact metric
manifolds has been recently investigated  by several authors.
Concerning coK\"{a}hler geometry, we mention the result that locally
conformal coK\"{a}hler manifolds admit a naturally defined
conformally invariant Weyl structure (\cite{matzeu2002}). Later on,
Matzeu proved that every $(2n + 1)$-dimensional coK\"{a}hler
manifold $(M,\phi,\xi,\eta, g)$ of constant $\phi$-sectional
curvature $c> 0$ admits two Ricci-flat Weyl structures where the
$1$-forms associated to the metric $g\in[g]$ are $\pm \theta = \pm
\lambda\eta$, where $\lambda=\frac{2c}{2n-1}$. More recently, she
generalized such a result by proving that if a compact coK\"{a}hler
manifold $(M,\phi,\xi,\eta, g)$ admits a closed Einstein-Weyl
structure $D^{[g]}$, then $M$ is necessarily $\eta$-Einstein
(\cite{matzeu2011}).

\medskip

A further well-known class of almost coK\"{a}hler manifolds satisfying some remarkable curvature condition is that of \emph{conformally flat almost
coK\"{a}hler manifolds}. In dimension greater than or equal to $5$ those manifolds are of non-positive scalar curvature and are coK\"{a}hler if and only if
they are locally flat (cf. \cite{olszak81}, \cite{olszak87}). Moreover, again under the assumption  $\dim(M)\geq 5$, any conformally flat almost
coK\"{a}hler manifold with K\"{a}hlerian leaves is necessarily coK\"{a}hler (\cite{dacko-olszak1998}). Contrary to that, in dimension $3$, there are
examples of conformally fiat almost coK\"{a}hler manifolds (which necessarily have K\"{a}hlerian leaves) that are not locally fiat and not coK\"{a}hler
(see again \cite{dacko-olszak1998} and \cite{dacko07}).

\small
\bibliography{survey}
\bibliographystyle{amsplain}
\end{document}